\documentclass{amsart}
\usepackage[utf8x]{inputenc}
\usepackage{amsmath, amsthm, amssymb}
\usepackage[usenames,dvipsnames,svgnames,table]{xcolor}
\usepackage[margin=1.3in]{geometry}
\usepackage{enumerate}
\usepackage{dsfont}
\usepackage{cleveref}
\usepackage{cite}
\usepackage{mathtools}
\usepackage[normalem]{ulem}
\usepackage{comment}
\usepackage{autonum}

\Crefname{Assumption}{Assumption}{Assumptions}
\Crefname{Theorem}{Theorem}{Theorems}
\Crefname{Lemma}{Lemma}{Lemmas}
\Crefname{Corollary}{Corollary}{Corollaries}
\Crefname{Proposition}{Proposition}{Propositions}
\Crefname{Theorem}{Theorem}{Theorems}
\Crefname{Conjecture}{Conjecture}{Conjectures}
\Crefname{Remark}{Remark}{Remarks}

\newtheorem{Theorem}{Theorem}[section]
\newtheorem{Proposition}[Theorem]{Proposition}
\newtheorem{Lemma}[Theorem]{Lemma}

\newtheorem{conjecture}[Theorem]{Conjecture}

\newcommand{\N}{\mathbb N}

\newcommand{\R}{\mathbb R}

\newcommand{\eps}{\varepsilon}

\DeclareMathOperator{\Tr}{tr}
\newcommand{\tr}{\Tr}
\DeclareMathOperator{\id}{Id}

\newcommand{\1}{\mathds{1}}
\newcommand{\vv}{\langle v\rangle}
\newcommand{\vvp}{\langle v'\rangle}

\newcommand{\calpha}{C_{x}^{\alpha/3}C_{v}^{\alpha}}
\newcommand{\calphamu}{C_{x}^{\mu\alpha/3}C_{v}^{\mu\alpha}}
\newcommand{\logalpha}{C_x^{\alpha/3} \log(1/C_v)^{-\theta}}
\newcommand{\logalphamu}{C_x^{\mu\alpha/3} \log(1/C_v)^{-\mu\theta}}
\newcommand{\logalphamup}{C_x^{\mu'\alpha/3} \log(1/C_v)^{-\mu'\theta}}

\newcommand{\logvalpha}{\log(1/C_v)^{-\theta}}

\newcommand{\vvo}{\langle v_0 \rangle}
\newcommand{\cP}{\mathcal P}

\renewcommand{\epsilon}{\eps}

\DeclareMathOperator{\Id}{Id}
\DeclareMathOperator{\supp}{supp}

\newcommand{\be}{\begin{equation}}
\newcommand{\ee}{\end{equation}}

\numberwithin{equation}{section}

\title[Time-irregular Schauder estimates and the Landau equation]{Kinetic Schauder estimates with time-irregular coefficients and uniqueness for the Landau equation}
\author{Christopher Henderson}
\address{Department of Mathematics, University of Arizona, Tucson, AZ 85721}
\email{ckhenderson@math.arizona.edu}
\author{Weinan Wang}
\address{Department of Mathematics, University of Arizona, Tucson, AZ 85721}
\email{weinanwang@math.arizona.edu}

\begin{document}

\begin{abstract}
	We prove a Schauder estimate for kinetic Fokker-Planck equations that requires only H\"older regularity in space and velocity but not in time.  As an application, we deduce a weak-strong uniqueness result of classical solutions to the spatially inhomogeneous Landau equation beginning from initial data having H\"older regularity in $x$ and only a logarithmic modulus of continuity in $v$.  This replaces an earlier result requiring H\"older continuity in both variables.
\end{abstract}

\maketitle

\section{Introduction}

This paper is concerned with the regularity of kinetic Fokker-Planck type equations of the form
\be\label{e.kfp_gen}
	(\partial_t + v\cdot\nabla_x) f
		= \tr(\bar a D^2_v f) + \bar b \cdot \nabla_v f + \bar c f
			+ g
				\qquad\text{ in } (0,T)\times \R^d \times \R^d
\ee
and the applications of this regularity theory to the Landau equation, which, roughly, is a fundamental model from gas dynamics for the evolution of a density of colliding particles~\cite{villani2002review, mouhot2018review}. 

Interest in the regularity of equations of the form~\eqref{e.kfp_gen} dates back to Kolmogorov~\cite{Kolmogorov1934}, who studied it with the choices $\bar a = \Id$, $\bar b = v$, and $\bar c = d$.  Kolmogorov explicitly computed the fundamental solution, which readily yields smoothing\footnote{It does not seem to be explicitly stated in~\cite{Kolmogorov1934} that Kolmogorov noticed the smoothing effect.  As a result, it is not clear when this ``hypoelliptic'' behavior was first identified in the simple setting Kolmogorov considered.} of $f$ in all variables despite only being elliptic in the $v$-variable.  We note two other computations of the fundamental solution in more general settings by Il'in~\cite{Ilin} and Weber~\cite{Weber}.  Eventually the observation that, in the setting of~\eqref{e.kfp_gen}, regularity in $v$ transfers to regularity in $x$ due to the transport operator $\partial_t + v\cdot\nabla_x$  led to H\"ormander's development of the theory of hypoellipticity~\cite{Hormander}.

Over the past few decades, a robust understanding of the role of the transport operator $\partial_t + v\cdot\nabla_x$ in regularity theory has been developed in the setting of Sobolev spaces.  The literature is truly vast, so we only cite a few prominent examples~\cite{villani_hypo,GLPS,Bouchut}; although, we encourage the reader to explore the references therein and the work that developed as a result.

More recently, there has been interest in precise quantitative estimates of regularity of solutions to kinetic equations in analogy with the regularity theory for parabolic equations.  In particular, the interest has been in the development of estimates in {\em continuity spaces}, such as H\"older spaces.  A suitable Harnack inequality has been proven~\cite{GuerandMouhot, GIMV, Wang-Zhang, Wang-Zhang-2, Zhu2021}, which yields the H\"older regularity of solutions to {\em divergence form} kinetic Fokker-Planck equations when the coefficients are merely bounded and elliptic-in-$v$ (note that~\eqref{e.kfp_gen} is in {\em non-divergence form}).  A Harnack inequality for non-divergence form kinetic operators remains elusive~\cite{SilvestreOpen}.

Additionally, Schauder estimates have been deduced and applied to various kinetic models~\cite{ImbertMouhot-toy, Manfredini, Polidoro, HS, bramanti2007schauder} (see also~\cite{Hao-2020, imbert2018schauder,ImbertSilvestre_survey} for estimates in the kinetic integro-differential setting).  These estimates yield bounds on higher H\"older regularity of the solution as long as the coefficients $\bar a$, $\bar b$, and $\bar c$ are H\"older continuous in all variables.

Our interest here is to investigate the minimal assumptions on the coefficients in~\eqref{e.kfp_gen} for proving the Schauder estimates.  As we detail below, this is inspired by the connection between this question and the conditions needed to prove uniqueness of solutions to the Landau equation.  Indeed, despite being nonlinear, the coefficients of the highest order terms in the Landau equation enjoy {\em better} regularity in $v$ than $f$ does.  It is, thus, natural to hope that only regularity in $v$ is necessary to prove the Schauder estimates.  We note that there are a number of related equations with similar structure to which the methods developed here may be applied: e.g., isotropic Landau~\cite{GualdaniZamponi-Isotropic-Landau,Isotropic-Landau, Gualdani-Guillen,Gualdani-review}, the Imbert-Mouhot toy model~\cite{Anceschi-Zhu,ImbertMouhot-toy}, and the Vlasov-Poisson-Landau equation~\cite{Guo-VPL}.

For parabolic equations, minimal assumptions  for the Schauder estimates similar to those considered here were first investigated by Brandt~\cite{Brandt}, who showed that H\"older regularity in $t$ is not necessary to establish partial Schauder estimates.  More precisely, one needs only have boundedness in $t$ and H\"older continuity in the spatial variables $x$ in order to show that $D^2_x f$ is H\"older continuous in $x$.  Knerr~\cite{Knerr} later strengthened this to deduce time regularity of $f$ under the same assumptions.  These two papers are the main inspiration for the present manuscript.  Their strategies are based on the comparison principle and are quite different from that used here, as we detail below.  There has been a large body of literature on this over the ensuing decades, see, e.g.,~\cite{DongJinZhang, DongSeick, Lieberman, SinestrarivonWahl}.

In this paper, we show that H\"older regularity in the time variable $t$ is not necessary to establish partial Schauder estimates for kinetic Fokker-Planck equations.  As an application of this, we deduce a weak-strong uniqueness result for classical solutions of the Landau equation starting from initial data that is $C^\alpha$ in $x$ and has a logarithmic modulus of continuity in $v$.  This improves upon an earlier uniqueness result in which H\"older regularity was required in both variables, and it indicates that the role of regularity in the uniqueness theory may be more technical than fundamental (although probably not nonexistent).  Below, we expand on this in detail and formalize a conjecture on less restrictive assumptions for uniqueness to hold.

We now make our main results more precise.

\subsection{Schauder estimates}

For simplicity, we consider the slightly less general equation
\be\label{e.kfp}
	(\partial_t + v\cdot\nabla_x) f
		= \tr(\bar a(t,x,v) D^2_v f) + \bar c(t,x,v) f
			+ g(t,x,v).
\ee
We note, however, that this essentially does not lose any generality.  This is discussed after the statement of the main result of this section \Cref{t.Schauder}.

We assume that $\bar a$ is uniformly elliptic and $\bar c$ and $g$ are bounded: there is $\Lambda>1$ such that
\be\label{e.ellipticity}
	\frac{1}{\Lambda} \Id
		\leq \bar a(t,x,v)
		\leq \Lambda \Id
	\qquad\text{and}\qquad
	|\bar c(t,x,v)|, |g(t,x,v)|
		\leq \Lambda.
\ee
We also assume that $\bar a$, $\bar c$, and $g$ are H\"older continuous in $(x,v)$: $\bar a, \bar c, g \in \calpha$.  The notation for this H\"older space is defined in \Cref{s.notation}.

Our first result is a general Schauder estimate that does not require the $t$-regularity of the coefficients.  Its proof is found in \Cref{s.Schauder}.

\begin{Theorem}\label{t.Schauder}
Fix $\alpha \in (0,1)$.  Assume that~\eqref{e.ellipticity} holds and $f, D^2_v f, \bar a, \bar c, g \in \calpha(Q_1)$. 
Then
\begin{equation}\label{e.w01151}
	\begin{split}
		[f]_{C_x^{(2+\alpha)/3}(Q_{1/2})}
		+ [D^2_v f]_{C_x^{\alpha/3}C^\alpha_v(Q_{1/2})}
		\lesssim
			&\left(1+ [c]_{\calpha(Q_1)} + [a]_{\calpha(Q_1)}^{1+ \frac{2}{\alpha}}\right) \|f\|_{L^\infty(Q_1)}
			\\&
				+ \left(1 + [a]_{\calpha(Q_1)}\right)
				[g]_{\calpha(Q_1)}.
	\end{split}
\end{equation}
The implied constant depends only on $d$, $\alpha$, and $\Lambda$.
\end{Theorem}

We note that a simple consequence of \Cref{t.Schauder} and~\eqref{e.kfp} is that
\be\label{e.c52408}
	\begin{split}
	[(\partial_t + v\cdot\nabla_x) f]_{\calpha(Q_{1/2})}
		\lesssim &\left(1+ [c]_{\calpha(Q_1)} + [a]_{\calpha(Q_1)}^{1+ \frac{2}{\alpha}}\right) \|f\|_{L^\infty(Q_1)}
			\\&
				+ \left(1 + [a]_{\calpha(Q_1)}\right)
				[g]_{\calpha(Q_1)}.
	\end{split}
\ee

Before commenting on the proof, we note that regularity in $t$ can easily be obtained at this point by two different methods.  The first is the hypoelliptic approach of \cite[Lemma~2.8]{ImbertMouhot-toy}.  The technique of the authors shows that shifts in $t$ decompose into a shifts in $v$ as well as shifts in transport (roughly, shifts according to the operator $\partial_t + v\cdot\nabla_x$).    The $v$-regularity is provided by \Cref{t.Schauder} and the transport regularity is provided by~\eqref{e.c52408}. 
The second approach is to notice that time regularity can easily be obtained  in the course of establishing \Cref{t.Schauder} with the same methods.  We did not opt for this due to (i) the desire for simplicity, (ii) the fact that time regularity does not play a role in our application (\Cref{t.Landau}), and (iii) the fact that the hypoelliptic approach of \cite[Lemma~2.8]{ImbertMouhot-toy} yields it in a simple manner as a consequence of \Cref{t.Schauder}.  Unfortunately, both approaches only provide H\"older continuity in $t$ of $f$ and do not yield regularity of $\partial_t f$.  For greater regularity in $t$, it appears that one needs more regularity of the coefficients.  We refer to~\cite{HS}.

\subsubsection{Strategy of the proof}

Our approach is along the lines of~\cite{HS}.  The proof proceeds with the same main two steps of every proof of Schauder estimates -- direct estimates for a ``homogeneous'' equation and then perturbing off of this ``homogeneous'' equation using the regularity of the coefficients.

The first step is slightly different from that of~\cite{HS}.  For us, the relevant homogeneous equation is the one where $\bar a(t,x,v) = \bar a(t)$.  Roughly, this allows us to perturb off of this case by using that, for $(x_0,v_0)$ fixed and $(x,v) \approx (x_0,v_0)$,
\be
	|\bar a(t,x,v) - \bar a(t,x_0,v_0)|
		\leq (|x-x_0|^{\alpha/3} + |v-v_0|^\alpha) [\bar a]_{\calpha}
		\ll 1.
\ee
Notice that this depends only on the regularity of $\bar a$ in $(x,v)$ and not in $t$.  We refer to this case, that is,~\eqref{e.kfp} when $\bar a$ does not depend on $(x,v)$ and $\bar c \equiv 0$, as the $(x,v)$-homogeneous equation.

It is worth discussing the $(x,v)$-homogeneous equation further.  Many proofs of the Schauder estimates for parabolic or kinetic equations (Brandt's \cite{Brandt} being a notable exception), hinge on the scaling in $t$ of moments of the fundamental solution $\Gamma_{\bar a}$, that is, integrals in $(x,v)$ of $\Gamma_{\bar a}$ with polynomial weights.   To obtain the estimates here, we compute the fundamental solution $\Gamma_{\bar a}$ explicitly (see \Cref{p.gamma_a}).  In~\cite{HS}, where the relevant homogeneous equation is $\bar a \equiv \Id$, it is essentially a basic calculus exercise to go from the explicit form of $\Gamma_{\Id}$ to the correct moment estimates.  In our setting, however, it is more difficult and requires a somewhat involved proof based on the dynamics of some matrix valued terms (see \Cref{l.w09032}).  Indeed, from \Cref{p.gamma_a} it is not even obvious that $\Gamma_{\bar a}$ is integrable in $(x,v)$.

The second step, that is,  the procedure of perturbing off of the homogeneous equation, proceeds as usual.

\subsubsection{Estimates for~\eqref{e.kfp_gen} versus~\eqref{e.kfp}}

As we mentioned above, there is essentially no loss in generality in considering~\eqref{e.kfp} in place of~\eqref{e.kfp_gen}.   The reason for this is that, one can obtain~\eqref{e.kfp} from~\eqref{e.kfp_gen} by letting
\be
	g_{\eqref{e.kfp}} = \bar b \cdot \nabla_v f + g_{\eqref{e.kfp_gen}}.
\ee
Here, to differentiate between the forcing term $g$ in~\eqref{e.kfp_gen} and the forcing term $g$ in~\eqref{e.kfp}, we use the equation number as the subscript.  
In this case, after applying \Cref{t.Schauder}, one has a $\calpha$-norm of $\nabla_v f$ on the right hand side of~\eqref{e.w01151}.  By interpolating, one can ``absorb'' this lower order term into the left hand side of~\eqref{e.w01151}.  While this is complicated by the different domains on which the norms are based, with $Q_{1/2}$ appearing on the left-hand side of~\eqref{e.w01151} and $Q_1$ appearing on the right-hand side of~\eqref{e.w01151}, it is generally possible to do, depending on the application.  The reader will surely have noticed that the same procedure should apply to the $\bar c$ term as well, and so the simplest presentation would consider only the case of~\eqref{e.kfp} with $\bar c \equiv 0$.  This is true; however, for the application we have in mind (\Cref{t.Landau}), it streamlines future computations to already have the explicit dependence on $[\bar c]_{\calpha}$.

\subsubsection{Further comments on related time irregular Schauder estimates}\label{s.further_Schauder}

As mentioned above, to our knowledge, the first result in this direction is due to Brandt~\cite{Brandt}, whose approach is entirely based on the comparison principle.  Indeed, in a very simple short paper, Brandt establishes Schauder estimates with precise dependence on the coefficients via the construction of an upper barrier for an appropriate finite difference of the solution to the parabolic equation under consideration.  Later, Knerr~\cite{Knerr} improved on the regularity obtained by Brandt by showing that, surprisingly, solutions had time regularity as well.  Knerr's strategy was also based on the comparison principle.

Unfortunately, despite~\eqref{e.kfp} enjoying a comparison principle, we were unable to adapt Brandt's strategy to the kinetic case.  We give a heuristic description of the obstruction.  One expects  ${\partial_t + v\cdot\nabla_x}$ in the kinetic case to act analogously to $\partial_t$ in the parabolic case.  A major difference, however, is that shifts in $t$ (the appropriate shifts related to $\partial_t$ regularity) have a directionality: time is one dimensional and there is a preferred direction, often called the ``arrow of time.''  Unfortunately, it is not clear what the analogue to this is in the kinetic setting with the operator $\partial_t + v\cdot\nabla_x$.  Very roughly, this is the roadblock to adapting Brandt's argument.  We note that this seems to be related to the impediment to proving a Harnack inequality for~\eqref{e.kfp_gen}  using the methods of Krylov and Safonov; see \cite[Section~8.2]{SilvestreOpen} for further discussion.

A few days prior to posting this manuscript to the arxiv, another very interesting paper was also posted by Biagi and Bramanti~\cite{BiagiBramanti} that investigates a similar problem to \Cref{t.Schauder}.  The authors consider ultraparabolic equations, a general class of equations that includes kinetic equations as a particular example, and they prove a Schauder estimate for time irregular coefficients.  Their proof proceeds along the same lines as ours; that is, they derive an explicit formula for the fundamental solution and use it to deduce the Schauder estimates.   
Their paper is focused entirely on the question of Schauder estimates of the form \Cref{t.Schauder} for a general family of ultraparabolic equations, and, as such, they do not consider applications of their theorem, as we do in \Cref{s.Landau_intro}.  Their work builds upon an earlier work of Bramanti and Polidoro~\cite{BramantiPolidoro} in which the fundamental solution of a class of ultaparabolic operators was studied in depth.  In particular, the authors construct it and establish that is has the appropriate regularity and Gaussian bounds.  We also mention connections to the other very recent preprint by Lucertini, Pagliarani, and Pascucci~\cite{LPP} in which the authors deduce optimal bounds on the higher regularity of the fundamental solution.  The estimates in~\cite{LPP} are strong enough to replace \Cref{l.w09031} in our proof of \Cref{t.Schauder}.  It seems likely that one could establish \Cref{t.Schauder} via an alternative approach to the Schauder estimates using the results of~\cite{LPP} directly.

\subsection{The Landau equation} \label{s.Landau_intro}

The Landau equation has the form:
\be\label{e.Landau}
	(\partial_t + v\cdot\nabla_x)f
		= \tr(\bar a^f D^2_v f) + \bar c^f f
			\qquad\text{ in } (0,T)\times \R^3\times\R^3,
\ee
where, for any function $h: \R^3 \to \R$,
\be\label{e.coefficients}
	\begin{split}
		&\bar a^h(t,x,v)
			= a_\gamma \int_{\R^3} \left(\Id - \frac{w\otimes w}{|w|^2}\right) |w|^{2+\gamma} h(v-w) \, dw
		\\&
		\bar c^h(t,x,v)
			= \begin{cases} c_\gamma \int_{\R^3} |w|^\gamma h(v-w) \, dw 
				\qquad&\text{ for } \gamma > -3,\\
			c_\gamma h
				\qquad&\text{ for } \gamma = -3.
			\end{cases}
	\end{split}
\ee
Here, $a_\gamma$ and $c_\gamma$ are positive constants whose exact value plays no role in the analysis and $\gamma \in [-3,0)$.  The physically relevant case is $\gamma=-3$.  We note that the regime $\gamma < 0$ is often called the {\em soft potentials} case and that the case $\gamma \in [0,1]$ is considered in many works, but we do not address it here.  We also note that~\eqref{e.Landau} is more often written in an equivalent divergence form, although that is not convenient for our work below.

We refer the reader to~\cite{villani2002review, mouhot2018review} for a general discussion of the Landau equation, its physical relevance, and its mathematical history.  We mention only that~\eqref{e.Landau} is nonlocal (that is, its coefficients at a point $(t,x,v)$ depend on the $f$ at other points $(t,x,v')$) and quasilinear (that is, the coefficient of the highest order term $\bar a^f$ depends on $f$).  As a result, the unconditional global well-posedness of classical solutions to~\eqref{e.Landau} is an extremely difficult problem that appears to be out of reach for the time being.

A new approach to this problem was initiated by Silvestre~\cite{Silvestre_Landau}, who proposed to study~\eqref{e.Landau} with methods coming from parabolic regularity theory under certain physically reasonable boundedness assumptions on the mass, energy, and entropy densities (see~\cite{ImbertSilvestre_survey} for a discussion of a similar program for the related Boltzmann equation).  We do not discuss these assumptions further, and we refer to this program as the {\em conditional regularity} program in the sequel.  These ideas have led to many new results, see, e.g.~\cite{GIMV, Cameron-Silvestre-Snelson, HS, ImbertMouhot-toy}.  The most relevant work to the present setting coming out of this program is~\cite{HST2019rough} that leveraged the ideas and theorems of the previous works~\cite{GIMV, Cameron-Silvestre-Snelson, HS} to obtain local well-posedness with fairly ``rough'' initial data.  In particular, the existence result in~\cite{HST2019rough} is in a weighted $L^\infty$ space, while the uniqueness result supposes, additionally, that the initial data is $C^\alpha$ for some $\alpha>0$.

One of the key insights used in the conditional regularity program is that, as previously mentioned, $\bar a^f$ enjoys better regularity in $v$ than $f$ does.  Indeed, $\bar a^f\in C^\alpha_v$ for any $\alpha \in (0,1)$ as long as $f$ is bounded and decays sufficiently quickly in $v$.   The gap between the existence and uniqueness results of of~\cite{HST2019rough}, described above, partially reflects the fact that the authors were not able to leverage this insight.  Our next result, a new uniqueness result for the Landau equation~\eqref{e.Landau}, provides a path in this direction.




\subsubsection{Uniqueness for the Landau equation}

We require the following non-degeneracy condition on $f_{\rm in}$:
there exist $r$, $\delta$, and $R>0$ so that
\be\label{e.nondegeneracy}
	 \text{for every } x\in \R^3, \text{ there is } v_x \in \R^3
	 \text{ such that } f_{\rm in}(x, \cdot) \geq \delta \1_{B_r(x,v_x)}.
\ee
The reason for~\eqref{e.nondegeneracy} is that, from it, one can obtain a pointwise lower bound for $f$.  This, in turn, yields the local-in-$v$ uniform ellipticity of $\bar a^f$.  This was originally shown in \cite[Theorem~1.3]{HST2018landau}; see also \cite[Lemma~2.5]{HST2019rough} for the connection between the lower bound on $f$ and the ellipticity of $\bar a^f$.

We are now ready to state the second main result.  It is proved in \Cref{s.Landau}.

\begin{Theorem}\label{t.Landau}
	Fix $k, \theta > 0$ and $\alpha \in (0,1)$.  Assume that $f_{\text{in}} \in \logalpha \cap L^{\infty,k}$ and satisfies~\eqref{e.nondegeneracy}.  Let $f \in L^{\infty,k}([0,T]\times \R^6)$ be any solution of~\eqref{e.Landau} constructed in \cite[Theorem 1.2]{HST2019rough} starting from initial data $f_{\rm in}$.
	
	Fix any uniformly continuous function $g \in L^{\infty, 5+\gamma+ \eta}([0,T]\times\R^6)$, where $\eta>0$, such that $g$ solves equation \eqref{e.Landau} weakly (in the sense of~\cite{HST2019rough}) and $g(t,x,v)\rightarrow f_{\text{in}}$ as $t\searrow 0$.
	
	Then, if $k$ is sufficiently large, depending on $\theta$, $\alpha$, and $\gamma$, and
	\be
		\frac{\theta}{2} \frac{\alpha}{2+\alpha} > 1,
	\ee
	there is $T_1 \in (0,T]$, depending only on $f_{\rm in}$, $\alpha$, $\theta$, $k$ and $\gamma$, such that $f=g$ in $[0,T_1]\times \R^6$.  If $k = \infty$ then $T_1 = T$.
\end{Theorem}
We note that the nonstandard continuity space $\log(1/C_v)^{-\theta}$ is defined in \Cref{s.notation} below, as are all of the notational conventions we use.  We also note that the particular type of weak solution plays almost no explicit role in our analysis since we immediately deduce various regularity properties of $g$ from previous results.  Hence, our choice of weak solution is made simply so that it is compatible with the previous results in~\cite{GIMV, HS}.  Roughly, though, $g$ is in an appropriate kinetic $H^1$-space and solves~\eqref{e.Landau} in the sense of integration against other kinetic $H^1$ test functions with compact support.

\subsubsection{The strategy of the proof}
We give a rough outline of the uniqueness argument used to prove \Cref{t.Landau}.  For simplicity, we ignore all complications due to ``weights'' in this discussion, although these are required in the proof due to the fact that $\bar a^f$ is only defined when $f$ decays sufficiently quickly as $|v|\to\infty$.  The proof follows the standard outline -- find an equation for the difference $f-g$ and use a Gronwall-type argument.  This, however, is complicated by the fact that~\eqref{e.Landau} is nonlocal and quasilinear, that is, the highest order coefficient is nonlinear in $f$.  As a result, we require an $L^\infty$ bound on $D^2_v f$.

We obtain such a $W^{2,\infty}_v$-bound by applying a scaled version of \Cref{t.Schauder}.  Were we to only assume that $f \in L^\infty$, such an estimate would degenerate like $1/t$ as $t\to0$ (recall that $f_{\rm in}\notin W_v^{2,\infty}$).  This can be seen easily by scaling arguments.  However, by propagating forward bounds on the $\logalpha$-norm, we are able to, via interpolation, obtain a bound that degenerates like
\be
	\|D^2_v f(t)\|_{L^\infty(\R^6)}
		\lesssim \frac{1}{t \left( \log\frac{1}{t}\right)^{\frac{\theta}{2}{\frac{\alpha}{2+\alpha}}}}.
\ee
Crucially, this is integrable in $t$ near $t=0$, which allows the Gronwall argument to close.

One key step above, scaling the Schauder estimates, was developed in~\cite{HS}.  The other key step above, in which we propagate the $\logalpha$-norm, relies on the general ideas of \cite[Proposition~4.4]{HST2019rough}, in which the $\calpha$-norm was propagated.  It is, however, significantly more complicated in our case.  The reason being that, while $\bar a^f$ is $v$-H\"older continuous, regardless of the regularity of $f$, $\bar c^f$ does not enjoy this property.  In particular, when $\gamma = -3$, which is the physically relevant case, $\bar c^f = c_\gamma f$.  As such, it is exactly as irregular as $f$.

Roughly, we overcome this by obtaining a bound on $\|D^2_v f(t_0)\|_{L^\infty(\R^6)}$ that depends on the $\logalpha$-norm of $f$ as well as on $\|D^2_v f\|_{L^\infty([t_0/4,t_0]\times \R^6)}$.  The appearance of this second term is exactly due to the (potential) irregularity of $\bar c^f$.  In this bound, the coefficient of  $\|D^2_v f\|_{L^\infty([t_0/4,t_0]\times \R^6)}$ is small.  Hence, by a careful argument, we are able to absorb it back into the $W_v^{2,\infty}$-term at $t_0$, despite the difference in time domains.  This step is contained in \Cref{p.w02091}.

The reason \Cref{t.Schauder} is useful in this application is the following.  In \cite{HST2019rough}, $(x,v)$-H\"older regularity of $f$ is propagated from initial $(x,v)$-H\"older regularity.  An additional argument shows that regularity is passed to $t$ as well.  Stated imprecisely, if $f \in \calpha$ then $f \in C^{\alpha/2}_t\calpha$.  The coefficients are then H\"older regular and the full Schauder estimates can be applied.  As we only have $\logalpha$-regularity of $f$, the only regularity that could potential be passed to $t$ is that with a $\log$ modulus.  At best, then, $\bar a^f$ will be H\"older in $(x,v)$, due to the H\"older regularity of $f$ in $x$ and the fact that $\bar a^f$ is defined by convolution with a ``nice'' kernel in $v$, but with only a $\log$ modulus of continuity in $t$.  Thus, the full Schauder estimates could not be applied.

We point out a subtle additional benefit to the application of \Cref{t.Schauder} in place of the Schauder estimates of~\cite{HS}.  When the estimates of~\cite{HS} are applied in~\cite{HST2019rough}, there is a loss of regularity between $\bar a^f$ and $f$ due to how time shifts interact with the appropriate notion of `kinetic distance\footnote{We have largely avoided discussing the kinetic distance since it plays no role in our analysis.  Indeed, without shifts in time, which we need not consider due to our not considering time regularity, the kinetic distance collapses to $\calpha$.   We point the interested reader to a clear discussion of the kinetic distance in~\cite[Section~2.1]{imbert2018schauder}.}'; see \cite[Lemma~2.7]{HST2019rough}.  This is avoided here due to our not considering time shifts.  As such, we achieve {\em sharper estimates} on the various quantities, such as $\|D^2_v f(t)\|_{L^\infty(\R^6)},$ as $t\searrow0$.

\subsubsection{Related work}

As mentioned above, \Cref{t.Landau} supercedes the earlier work~\cite[Theorem~1.4]{HST2019rough}, which required $(x,v)$-H\"older regularity of $f_{\rm in}$ for uniqueness the hold.  
We also mention the work of Anceschi and Zhu in~\cite{Anceschi-Zhu} on a similar model.

To our knowledge, the local well-posedness theory of Landau is relatively unstudied, with more interest directed toward a related kinetic integro-differential equation, the Boltzmann equation.  There, the first local well-posedness results are due to the AMUXY group~\cite{amuxy2010regularizing,amuxy1,amuxy2}.  In particular, a general uniqueness result in an appropriate Sobolev space (of order $2s$, that is, twice the order of the differential operator in the equation) was proven in~\cite{amuxy2011uniqueness} for a restricted class of parameters.

While~\cite{amuxy2011uniqueness} is an extremely nice result, we describe in slightly more detail its limitations in order to highlight the difficulties in our setting.  Their result requires $H^{2s}_v$-regularity of solutions.  The Landau equation essentially corresponds to the $s=1$ case.  Were their result to apply, it would require $H^2_v$-regularity of $f$, which corresponds to $C^{1/2}_v$-regularity and is significantly more than we require here.  Indeed, for reasons related to this, we note that uniqueness is, in some ways, more difficult for Landau than Boltzmann as the differential operator is of higher order.  Additionally, their uniqueness result requires regularity of both solutions in contrast to our result that has only mild conditions on the other potential solution $g$.  On the other hand, their result only requires boundedness in $x$.

The close-to-equilibrium and homogeneous setting for~\eqref{e.Landau} have seen more focus.  This is probably due to the fact that one is often able to establish strong results such as global well-posedness and convergence to equilibrium.   The state-of-the-art, techniques, and types of questions asked in these settings are quite different from those raised in the current manuscript, so we do not go into much detail here.  We simply mention a few landmark results in each case.  The story in the homogeneous setting (that is, when $f$ is independent of $x$) is somewhat complicated by the functional setting one works in, but we mention the works of~\cite{FournierGuerin,Alexandre-Liao-Lin, Wu, Fournier-CMP, Desvillettes-Villani-1, Desvillettes-Villani-2}.  In the close-to-equilibrium setting (that is, when $f_{\rm in}$ is ``close'' to a Maxwellian of the form $\alpha e^{-|v|^2/\beta}$ for some $\alpha,\beta>0$), we refer to~\cite{Guo-2002, Mouhot-Neumann, villani-96, DLSS}.  Both settings are extremely well studied and, as a result, we are only able to reference a small selection of the work completed over the past several decades.  

Outside of these settings, little is known about the global well-posedness of classical solutions.  To our knowledge, the conditional result of~\cite{HST2019rough}, which yields global well-posedness as long as the mass and energy densities remain bounded in $t$ and $x$ in the case $\gamma > -2$ (or, in the case of $\gamma \leq -2$, if certain $L^p$-norms remain bounded), is currently the sharpest condition ruling out ``blow-up.''



\subsection{Two conjectures}

We now formulate two conjectures regarding ways in which the results above might be strengthened.

First, if we trust the analogy discussed above, that $\partial_t + v\cdot\nabla_x$ in the kinetic setting is similar to $\partial_t$ in the parabolic one, we are led to the following conjecture:
\begin{conjecture}
	Fix any $\alpha \in (0,1)$.  Assume that~\eqref{e.ellipticity} holds and that $f$, $D^2_vf$, $\bar a$, $\bar c$, $g \in C_v^\alpha(Q_1)$.  Then
\be
	\begin{split}
		[D^2_v f]_{C^\alpha_v(Q_{1/2})}
		\lesssim
			&\left(1+ [c]_{C^\alpha_v(Q_1)} + [a]_{C^\alpha_v(Q_1)}^{1+ \frac{2}{\alpha}}\right) \|f\|_{L^\infty(Q_1)}
				+ \left(1 + [a]_{C^\alpha_v(Q_1)}\right)
				[g]_{C^\alpha_v(Q_1)}.
	\end{split}
\ee
The implied constant depends only on $d$, $\alpha$, and $\Lambda$.
\end{conjecture}

Notice that the conjectured result above does not require any $x$-regularity.  It seems that a uniqueness result for the Landau equation is an immediate consequence of this.  We state this roughly here:
\begin{conjecture}
	In the setting of \Cref{t.Landau}, although assuming only that $f_{\rm in} \in \log(1/C_v)^{-\theta} \cap L^{\infty,k}$ (that is, we drop the H\"older regularity in $x$), the same weak-strong uniqueness conclusion holds as long as
	\be
		\theta > 2.
	\ee
\end{conjecture}

It is not clear that the above conjecture, were it true, would be sharp.  There is a strong connection between regularity and uniqueness results. Indeed, recent work has established the nonuniqueness of irregular (weak) solutions of fluid equations, see, e.g.,~\cite{Camillo,Buckmaster-Vicol}.  We also note the work of Kiselev, Nazarov, and Shterenberg, who, in the critical case of the fractal Burgers equation studied in~\cite{KiselevNazarovShterenberg}, see a situation similar to that of the Landau equation: rough solutions immediately become smooth but uniqueness is unknown without further regularity assumptions.  In fact, despite the intense interest in~\cite{KiselevNazarovShterenberg}, uniqueness of these rough solutions remains open as far as we know.

On the other hand, in the homogeneous ($x$-independent) case for the Landau equation, where the Landau equation has more structure, uniqueness has been established through a probabilistic approach that yields bounds on the Wasserstein distance between two solutions~\cite{FournierGuerin,Fournier-CMP}.  This result requires essentially no regularity of $f$, although it is only applicable in the homogeneous case.

We expect the conjectures above to be difficult to establish for reasons related to the fundamental difference between $\partial_t$ in the parabolic setting and $\partial_t + v\cdot\nabla_x$ in the kinetic setting that were discussed in \Cref{s.further_Schauder}.

\subsection{Notation and continuity spaces}\label{s.notation}

\subsubsection{Points and kinetic cylinders}

%
For succinctness, we often write
\be
	z = (t,x,v),
	\quad
	\tilde z = (\tilde t, \tilde x, \tilde v),
		\quad\text{and}\quad
	z' = (t',x',v').
\ee
For any $r>0$, we let
\be
	Q_r
		= (-r^2, 0] \times B_{r^3} \times B_r,
\ee
where we use the convention that if the base point of a ball is not stated then it is the origin; that is
\be
	B_r = B_r(0).
\ee
The reason for the choice of $Q_r$ is the natural scaling $(t,x,v) \mapsto (r^2 t, r^3 x, rv)$ associated to~\eqref{e.kfp}.

\subsubsection{Continuity spaces}

Throughout we work with some inhomogeneous continuity spaces, i.e., those in which different `amounts' of regularity are required in each variable.  In particular, for a any set $Q \subset \R\times \R^d \times \R^d$ and parameters $\alpha_1,\alpha_2 \in (0,1]$, we let 
\be
	C^{\alpha_1}_x C^{\alpha_2}_v(Q)
		:= \{f : Q \to \R: f\in L^\infty(Q), [f]_{C^{\alpha_1}_x C^{\alpha_2}_v(Q)} < \infty\},
\ee
where
%
%
%
%
%
\be
	[f]_{C^{\alpha_1}_x C^{\alpha_2}_v(Q)}
		:=
		\sup_{\substack{(t,x,v) \neq (t,x',v') \in Q,\\ |x-x'|, |v-v'| < 1/2}} \frac{|f(t,x,v) - f(t',x',v')|}{|x- x'|^{\alpha_1} + |v- v'|^{\alpha_2}}.
\ee



Finally, for the uniqueness result for the Landau equation, we define a space of functions whose modulus of continuity is logarithmic.  Indeed, for $Q \subset \R\times \R^d \times \R^d$ and parameters $\alpha \in (0,1)$ and $\theta > 0$, we let
\be
	\logalpha(Q)
		:= \{f : Q \to \R: f\in L^\infty(Q), [f]_{\logalpha(Q)} < \infty\},
\ee
where
\be\label{e.logalpha}
	[f]_{\logalpha(Q)}
		= \sup_{\substack{(t,x,v) \neq (t,x',v') \in Q,\\ |x-x'|, |v-v'| < 1/2}} \frac{|f(t,x,v) - f(t,x',v')|}{|x- x'|^{\alpha/3} + \log(1/|v- v'|)^{-\theta}}.
\ee
Abusing notation, we also use the $\logalpha$ notation for functions $f$ that are independent of $t$ but for which the supremum in~\eqref{e.logalpha}, without the $t$ terms, is finite.

When $Q$ is not specified in the norms above, it is taken to be either $\R^6$ or $\R_+ \times \R^6$, depending on the setting.  For example, if $f: \R^6 \to \R$, then we say $f\in \calpha$ to mean $f\in \calpha(\R^6)$.  

\subsubsection{Multi-indices}

Given a multi-index $\alpha \in (\N\cup\{0\})^d$, we write
\be
	\partial_v^\alpha
		= \partial_{v_1}^{\alpha_1}\cdots \partial_{v_d}^{\alpha_d}.
\ee
The object $\partial_x^\alpha$ is defined analogously.

%
%
%
%

\subsubsection{Other notation}\label{sec:notation}

Throughout the work, constants are assumed to change line-by-line and depend on various parameters such as the dimension $d$, the ellipticity constant $\Lambda$, and the regularity parameter $\alpha$.  In the statement of each result, we make clear the dependencies and in its proof, we simply write $A \lesssim B$ when $A \leq C B$, where $C$ is a constant depending on those parameters.  We use $A\approx B$ when $A\lesssim B$ and $B\lesssim A$.

In the uniqueness result for the Landau equation, we must work with weighted spaces.  To this end, we recall the Japanese bracket: for any $v \in \R^d$,
\be
	\vv = \sqrt{1 + |v|^2}.
\ee
Then we define the associated weighted $L^\infty$-spaces: for any $n$,
\be
	L^{\infty,n}
		:= \{f : \vv^n f \in L^\infty\}
	\qquad\text{with norm }
	\|f\|_{L^{\infty,n}}
		:= \| \langle\cdot\rangle^n f\|_{L^\infty}.
\ee

\section{The Schauder estimates}\label{s.Schauder}

In this section, we prove our first main result \Cref{t.Schauder}, which is the Schauder estimates for~\eqref{e.kfp}.  As usual, the proof proceeds in two steps.  The first step is an estimate for solutions of a relevant `homogenous' equation.  For us, this `homogeneous' equation is the one where the coefficients depend only on time $t$.  The second step (\Cref{s.w05171}) is to bootstrap to the general case by perturbing off of this `homogeneous' equation.

\subsection{The first step: Schauder estimates for the $(x,v)$-homogeneous problem}

Consider the basic kinetic Fokker-Planck equation involving only transport in $x$ and diffusion in $v$ which the diffusion has a coefficient depending only on $t$:
\begin{equation}\label{e.homogeneous_kfp}
	\begin{split}
		(\partial_{t} + v\cdot \nabla_{x})f
		=
		\tr(\bar a(t)D^{2}_vf)
		+ g.
	\end{split}
\end{equation}
Our assumption on $\bar a$ is the following: there is $\Lambda\geq 1$ such that
\be\label{e.homogeneous_assumptions}
		\bar a: \R \to \R^3\times\R^3
		\text{ is measurable and }
		\quad
		\frac{1}{\Lambda} \id
			\leq \bar a(t)
			\leq \Lambda \id
			\quad\text{ for all $t$}.
\ee
We stress that $\bar a$ does not satisfy any further regularity assumptions.

We begin by studying the fundamental solution of this problem; that is, the function $\Gamma_{\bar a}$ for which the solution $f$ of~\eqref{e.homogeneous_kfp} is given by 
\be\label{e.kernel}
	\begin{split}
	f(z)
		= 
			\int_\R \int_{\R^d} \int_{\R^d}
				\Gamma_{\bar a}(t, x- \tilde x - (t-\tilde t)\tilde v, v-\tilde v; \tilde t)g(\tilde t,\tilde x,\tilde v)\, d\tilde z.
	\end{split}
\ee
In the simple case $\bar a \equiv \id$, it is well-known~\cite{Hormander} that $\Gamma_{\id}$ is given by
\be\label{e.gamma_id}
	\Gamma_{\id}(t,x,v)
		= \begin{cases}
			\left(\frac{\sqrt 3}{2 \pi t^2}\right)^d
				\exp\Big\{
					- \frac{3|x-vt/2|^2}{t^3}
					- \frac{|v|^2}{4t}
				\Big\}
					\qquad&\text{ if } t>0,\\
			0
				\qquad&\text{ if } t\leq 0.
		\end{cases}
\ee
We point out two features: (1) integrating $\Gamma_{\id}$ in $x$ recovers the standard heat kernel, and (2) the `kinetic convolution' involved in~\eqref{e.kernel} respects the Galilean invariance induced by the transport operator.

By a somewhat complicated, but nonetheless straightforward, Fourier transform-based computation we can compute the fundamental solution associated to a general $(x,v)$-independent $\bar a$.  Indeed, we find the following:
\begin{Proposition}\label{p.gamma_a}
Under the assumption~\eqref{e.gamma_id}, solutions of~\eqref{e.homogeneous_kfp} are given by~\eqref{e.kernel} with the fundamental solution
\be\label{e.gamma_a}
	\Gamma_{\bar a}(t,x,v;s)
		=
		\begin{cases}
		\frac{\exp\Big\{-\frac{v\cdot A_0(t;s)^{-1}v}{4}
			-
				(x - M(t;s) v)\cdot  P(t;s)^{-1} (x - M(t;s) v)
			\Big\}}{(4\pi)^d\sqrt{\det{A_0(t;s)P(t;s)}}}
			\qquad&\text{ if } t >s,\\
		0
			\qquad&\text{ if } t\leq s,
		\end{cases}
\ee
	where
\begin{equation}\label{e.A_i}
	\begin{split}
		&A_i(t;s)=\int_{s}^{t} (r-s)^i \bar a(r)\,dr,
			\qquad
		 \text{for}~i=0, 1, 2,
		 \\&
		P(t;s)
		=
		A_{2}(t;s)-A_{1}(t;s)A_0(t;s)^{-1}A_{1}(t;s),
		\qquad\text{and}
	\\&
		M(t;s)
			= (t-s)\id - A_0(t;s)^{-1} A_1(t;s).
	\end{split}
\end{equation}
\end{Proposition}
\noindent We postpone the proof of \Cref{p.gamma_a} to \Cref{appendix}.

It is not obvious that~\eqref{e.kernel} is well-defined from~\eqref{e.gamma_a}.  Indeed, while the positive-definiteness of $A_0$ and its having the same scaling in time as the analogous term in~\eqref{e.gamma_id} are immediately obvious, the same cannot be said for $P(t)$.  In fact, even the positive definiteness of $P$ is not clear.  However, we need a stronger estimate than merely positive definiteness of $P$ as the crucial step in most proofs of the Schauder estimates is in understanding the scaling in $t$ of $\Gamma_{\bar a}$ and its integrals in $x$ and $v$.

We now state this scaling property.  Notice that it is, up to constants, the same as one would obtain using $\Gamma_{\id}$ defined in~\eqref{e.gamma_id}.  Its proof is contained in~\Cref{s.kernel_scaling}.
\begin{Lemma}\label{l.w09031}
	Let $\Gamma_{\bar a}$ be as in~\eqref{e.gamma_a}, with $\bar a$ under the assumptions given by~\eqref{e.homogeneous_assumptions}. 
	Fix any multi-indices $\alpha, \beta \in (\N\cup\{0\})^d$, any natural number $j\geq 0$, and any real numbers $r,s \geq 0$.
	For $t > \tilde t$,
\begin{equation}\label{e.w09161}
	\begin{split}
		\int_{\R^d}\int_{\R^d}
			\max_{(0,\xi_2,\xi_3) \in Q_{(t-\tilde t)/2}}
			|\partial^{j}_{t}\partial_{x}^{\beta}\partial_{v}^{\alpha}\Gamma_{\bar a}(t,x+\xi_2,v+\xi_3; \tilde t)|x|^{r}|v|^{s}\,dxdv
		\lesssim
		(t-\tilde t)^{-\frac{2j+|\alpha|+3|\beta|}{2}+\frac{3r+s}{2}}
	.
	\end{split}
\end{equation}
\end{Lemma}

Using this estimate, \Cref{l.w09031}, we are now able to establish the main result in the $(x,v)$-homogeneous setting that will be the basis of the main Schauder estimate. 

\begin{Proposition}\label{l.w12041}
	Fix $\alpha \in (0,1)$.  Suppose that $f$, $(\partial_t + v\cdot\nabla_x) f$, $D^2_v f$, and $g \in C_x^{\alpha/3}C_v^{\alpha}(Q_1)$.  Assume that $f$ and $g$ have compact support in $Q_1$ and satisfy~\eqref{e.homogeneous_kfp} with coefficient $\bar a$ that satisfies~\eqref{e.homogeneous_assumptions}.  Then
	\be
	[f]_{C_x^{(2+\alpha)/3}(Q_1)}
	+
	[D_{v}^{2}f]_{\calpha(Q_1)}
	\lesssim
	[g]_{C^{\alpha/3}_xC^\alpha_v(Q_1)},
	\ee
	where the implied constants depend only on $\alpha$, $\Lambda$, and the dimension $d$.
\end{Proposition}
\begin{proof}
	We begin by estimating $[D_{v}^{2}f]_{\calpha(Q_1)}$.  Recalling~\eqref{e.kernel}, for any $z=(t,x,v) \in  Q_{1}$ and any $1\leq i, j\leq d$,
	\be
	\begin{split}
		\partial_{v_{i}v_{j}}f(z)
			&=
			\int_{-1}^{t} \int_{\R^d}\int_{\R^d}
			\partial_{v_{i}v_{j}}\Gamma_{\bar a}(t, x-\tilde x - (t-\tilde t)\tilde v, v - \tilde v; \tilde t)g(\tilde z)\,d\tilde z
	.
	\end{split}
	\ee

	Fix another point $z'\in Q_1$ of the form
	\be \notag
		z' = (t, x', v').
	\ee  
	Notice that $z$ and $z'$ have the same $t$-coordinate.  This is due to the fact that we do not prove any regularity of $D^2_v f$ in $t$.  
	Let
\be
\begin{split}\notag
	&h
		= |x-x'|^{1/3} + |v-v'|
		\quad\text{ and }
	\\
	&\delta g(\tilde z)
		= g(t - \tilde t, x - \tilde x - \tilde t(v-\tilde v), v - \tilde v)
			- g(t - \tilde t, x' - \tilde x - \tilde t(v' - \tilde v), v' - \tilde v).
\end{split}
\ee
	Then, after making the change of variables
	\be
		\tilde z \mapsto (t-\tilde t, x - \tilde x - (t-\tilde t)\tilde v, v - \tilde v),
	\ee
	we find
	\begin{equation}\notag
	\begin{split}
	\partial_{v_{i}v_{j}}f(z)-\partial_{v_{i}v_{j}}f(z')
	&=
	\left(\int_{0}^{2h^2}+\int_{2h^2}^{1+t}\right) \int_{\R^d}\int_{\R^d}
	\partial_{v_{i}v_{j}}\Gamma_{\bar a}(t, \tilde x, \tilde v; t-\tilde t)\delta g(\tilde z)\, d\tilde z
	\\&=
	I_1 +I_2.
	\end{split}
	\end{equation}
	We now estimate each of $I_1$ and $I_2$ in turn.
	
	\medskip
	{\bf Estimating $I_1$:} Here, integrating $\partial_{v_iv_j}\Gamma_{\bar a}$ over $(\tilde x, \tilde v)$ leaves us with an $O(1/\tilde t)$ term.  This means that our approach needs to use the regularity of $g$ to obtain $\tilde t$-terms, either directly or via \Cref{l.w09031}.  Using the regularity of $\partial_{v_iv_j}\Gamma_{\bar a}$ will only exacerbate this issue, so we do not use it, but obtain extra smallness instead by working on a small interval $[0,2h^2]$.

	To this end, we smuggle in a new term.  Setting $\tilde z_{0}=(\tilde t, \tilde x, 0)$, we see that
	\be\notag
		\int_0^{2h^2} \int_{\R^d}\int_{\R^d} \partial_{v_iv_j} \Gamma_{\bar a}(t, \tilde x, \tilde v; t- \tilde t) \delta g(\tilde z_0)\, d\tilde z
		= 0.
	\ee
	Hence, we obtain
	\be\notag
	|I_1|
	=
	\Big|	\int_{0}^{2h^2} \int_{\R^d}\int_{\R^d}
	\partial_{v_{i}v_{j}}\Gamma_{\bar a}(t, \tilde x, \tilde v; t- \tilde t)(\delta g(\tilde z)-\delta g(\tilde z_{0}))\,d\tilde z \Big|.
	\ee
	Next, we point out that
	\be\label{e.c032901}
	\begin{split}
		|\delta g(\tilde z) - \delta g(\tilde z_0)|
			&= \big|
				\left(
					g(t-\tilde t, x - \tilde x - \tilde t(v - \tilde v), v - \tilde v) - g(t-\tilde t, x' - \tilde x - \tilde t(v' - \tilde v), v' - \tilde v)
				\right)
			\\&\qquad
				- \left(
					g(t-\tilde t, x - \tilde x - \tilde t v , v) - g(t-\tilde t, x' - \tilde x - \tilde tv', v')
				\right)\big|
			\\&
			= \big|
				\left(
					g(t-\tilde t, x - \tilde x - \tilde t(v - \tilde v), v - \tilde v) - g(t-\tilde t, x - \tilde x - \tilde t v , v)
				\right)
			\\&\qquad
				- \left(g(t-\tilde t, x' - \tilde x - \tilde t(v' - \tilde v), v' - \tilde v)
					 - g(t-\tilde t, x' - \tilde x - \tilde tv', v')
				\right)\big|
			\\&
			\leq 2[g]_{C^{\alpha/3}_xC^\alpha_v(Q_1)}
				\left( (\tilde t |\tilde v|)^{\alpha/3}
					+ |\tilde v|^\alpha\right),
	\end{split}
	\ee
	where, to get the second inequality, we swapped the places of the second and third terms in the absolute values.  This has the advantage of avoiding a $t-t'$ term that would require time regularity of $g$.
	
	Using~\eqref{e.c032901} and then \Cref{l.w09031}, we find
	\be
	\begin{split}\label{e.w05011}
	|I_1|
	&\lesssim
	[g]_{C^{\alpha/3}_x C^{\alpha}_v(Q_1)} 
	\int_{0}^{2h^2} \int_{\R^d}\int_{\R^d}
	|\partial_{v_{i}v_{j}}\Gamma_{\bar a}(t, \tilde x, \tilde v; t- \tilde t)|((\tilde t |\tilde v|)^{\alpha/3}+|\tilde v|^{\alpha})\,d\tilde z
	\\&\lesssim
	[g]_{C^{\alpha/3}_x C^{\alpha}_v(Q_1)} 
	\int_{0}^{2h^2} s^{\alpha/2-1}\,ds
	\lesssim
	[g]_{\calpha(Q_1)} h^{\alpha}
	.
	\end{split}
	\end{equation}

	\medskip

	{\bf Estimating $I_2$:}  
%
%
	In this case, we are insulated from $\tilde t = 0$ so we may (and do) use the regularity of $\partial_{v_iv_j}\Gamma_{\bar a}$ here. 
%
%
%
	The first step is to separate the two integrals in $I_2$ (recall that $\delta g$ is a difference of two terms and then change variables $\tilde z \mapsto z - \tilde z$ and $\tilde z \mapsto z' - \tilde z$, respectively.  This yields:
	\begin{equation}\notag
	\begin{split}
	I_2
	&=
	\int_{-1}^{t-2h^2} \int_{\R^d}\int_{\R^d}
	(\partial_{v_{i}v_{j}}\Gamma_{\bar a}(t, x - \tilde x, v -\tilde v; \tilde t)-\partial_{v_{i}v_{j}}\Gamma_{\bar a}(t, x' - \tilde x, v' -\tilde v; \tilde t))
	g(\tilde t,\tilde x-(t-\tilde t)\tilde v,\tilde v)\,d\tilde z.
	\end{split}
	\end{equation}
	The key reason for doing this is so that the resulting terms, $\partial_{v_iv_j}\Gamma_{\bar a}$, are a full $\tilde v$-derivative.  Hence,
    	\begin{equation}\notag
    		\int_{\R^d}
    	\partial_{v_{i}v_{j}}\Gamma_{\bar a}(t,x-\tilde x,v-\tilde v; \tilde t) \,d\tilde v
	= 
	\int_{\R^d}\partial_{v_{i}v_{j}}\Gamma_{\bar a}(t,x'-\tilde x,v'-\tilde v; \tilde t)\,d\tilde v
    	=
    	0.
    	\end{equation}
    	Therefore, we rewrite $I_{2}$ as 
    	\begin{equation}
    	\begin{split}
	I_2
    	&=
    		\int_{-1}^{t-2h^2} \int_{\R^d}\int_{\R^d}
    	(\partial_{v_{i}v_{j}}\Gamma_{\bar a}(t, x-\tilde x, v - \tilde v; \tilde t)-\partial_{v_{i}v_{j}}\Gamma_{\bar a}(t, x'-\tilde x, v' - \tilde v; \tilde t)
    	\\&\phantom{MMMMMMMa}\times
    	(g(\tilde t,\tilde x-(t-\tilde t)\tilde v,\tilde v)-g(\tilde t,\tilde x-(t-\tilde t)v,v))\,d\tilde z.
    	\end{split}
    	\end{equation}
	Notice that $z - z' \in Q_h$.  By a Taylor approximation, we see that
    	\begin{equation}\notag
    	\begin{split}
    	&|\partial_{v_{i}v_{j}}\Gamma_{\bar a}(t, x-\tilde x, v - \tilde v; \tilde t)-\partial_{v_{i}v_{j}}\Gamma_{\bar a}(t, x'-\tilde x, v' - \tilde v; \tilde t)|
    	\\&\leq
	    	\max_{\xi \in Q_h,\xi_1=0}\left(
    	h^3 |\nabla_x \partial_{v_{i}v_{j}}\Gamma_{\bar a}(t, x-\tilde x + \xi_2, v - \tilde v + \xi_3; \tilde t)|
    	+
    	h |\nabla_v \partial_{v_{i}v_{j}}\Gamma_{\bar a}(t, x-\tilde x + \xi_2, v - \tilde v + \xi_3; \tilde t)|
    	\right).
    	\end{split}
    	\end{equation}
	Additionally, we have
	\be
		|g(\tilde t,\tilde x-(t-\tilde t)\tilde v,\tilde v)-g(\tilde t,\tilde x-(t-\tilde t)v,v)|
			\leq [g]_{\calpha(Q_1)} [(|t-\tilde t||\tilde v - v|)^{\alpha/3}
				+ |\tilde v - v|^\alpha].
	\ee

    	Therefore,  by a shifting back in all variables, $(\tilde x, \tilde v) \mapsto(x-\tilde x, v - \tilde v)$, 
	we see
    	\begin{equation}\notag
    	\begin{split}
	&|I_2|
    	\leq
    	[g]_{\calpha(Q_1)}
    	\int_{2h^2}^{t+1} \int_{\R^d}\int_{\R^d}
	\max_{\xi \in Q_h,\xi_1=0}\Big(
    	h^3 |\nabla_x \partial_{v_{i}v_{j}}\Gamma_{\bar a}(t, \tilde x + \xi_2, \tilde v + \xi_3; \tilde t)|
    	\\&
	\phantom{MMMMMMMMMMMM}+
    	h |\nabla_v \partial_{v_{i}v_{j}}\Gamma_{\bar a}(t, \tilde x + \xi_2,  \tilde v + \xi_3; \tilde t)|
    	\Big)
    	(|\tilde t|^{\alpha/3}|\tilde v|^{\alpha/3}
    	+
    	|\tilde v|^{\alpha})
    	\,d\tilde z.
    	\end{split}
    	\end{equation}
	Using then \Cref{l.w09031}, which, effectively turns $\tilde x$ and $\tilde v$ into $\tilde t^{3/2}$ and $\tilde t^{1/2}$, respectively, and $\partial_t$, $\nabla_x$ and $\nabla_v$ into $\tilde t^{-1}$, $\tilde t^{-3/2}$, and $\tilde t^{-1/2}$, respectively, we find
	\be\label{e.c050401}
		\begin{split}
		|I_2|
			&\lesssim [g]_{\calpha(Q_1)} \int_{2h^2}^{t+1} \Big(
			\frac{h^2}{\tilde t^2} + \frac{h^3}{\tilde t^{5/2}}
				+ \frac{h}{\tilde t^{3/2}}\Big)
				\tilde t^{\alpha/2}
				\, d\tilde t
			\lesssim
				[g]_{\calpha(Q_1)} h^\alpha.
		\end{split}
	\ee

	Combining this,~\eqref{e.w05011}, and~\eqref{e.c050401} completes the estimate of $[D_{v}^{2}f]_{\calpha(Q_1)}$ as claimed in the statement.
	

	The estimate of $[f]_{C^{\frac{2+\alpha}{3}}_x(Q_1)}$ essentially proceeds along the same lines, but is significantly simpler as there is no difference in the Galilean terms $-\tilde t(v - \tilde v)$ in $\delta g$.  Additionally, the details are exactly the same as in \cite[Lemma~2.5]{HS}.  As such we omit the proof.
%
%
%
%
	\end{proof}

\subsubsection{Proof of \Cref{l.w09031}: integrals of $\Gamma_{\bar a}$ and their scaling in $t$}\label{s.kernel_scaling}

Our first observation in service of establishing \Cref{l.w09031} is that the integral is well-defined due to the positivity of the exponential terms and that it satisfies the appropriate scaling laws.
%
%
%
\begin{Lemma}\label{l.w09032}
	The matrix $P(t)$ is invertible.  Additionally, we have the following bounds: for all $t>\tilde t$, any $i= 0,1,2$, and any vector $w\in\R^d$,
	\be\notag
		\begin{aligned}
		(i) & \quad w\cdot P(t;\tilde t)w
			\approx (t-\tilde t)^3 |w|
		&\qquad\qquad
		(ii) & \quad |A_i(t;\tilde t)w|
			\approx (t-\tilde t)^{i+1}|w|
		\\
		(iii) & \quad |M(t;\tilde t)w|
			\lesssim (t-\tilde t)|w|.
		\end{aligned}
	\ee
	where the constants depends only on $\Lambda$.
\end{Lemma}
We note that the lower bound in $(ii)$ and all of the upper bounds are straightforward, but the lower bound in (i) is not obvious and nontrivial to prove.  As the proof is somewhat long, we postpone it to \Cref{s.matrices}.

Next, we observe that the partial derivative of $\Gamma$ appearing in \Cref{l.w09031} has a particular form.
\begin{Lemma}\label{l.w09171}
	Fix any multi-indices $\alpha$ and $\beta$ as in \Cref{l.w09031}. Then there exist a homogeneous polynomial $\cP_{\alpha,\beta}$ of order $|\alpha| + 3|\beta|$ such that
	\be\label{e.c9164}
	\begin{split}
		\frac{\partial_x^\beta \partial_v^\alpha \Gamma_{\bar a}(t,x,v\tilde t)}{\Gamma_{\bar a}(t,x,v;\tilde t)}
		&= \cP_{\alpha,\beta}(
		(A_0^{-1/2})_{ij},
		(P^{-1/2} M)_{ij},
		(A_0^{-1}v)_j,
		(M^T P^{-1} M v)_j,
		\\&\qquad\qquad
		(M^T P^{-1} x)_j,
		P^{-1/6},
		(M^T P^{-1})^{1/3}
		),
	\end{split}
	\ee
	where the last two terms in the polynomial, $(M^T P^{-1})^{1/3}$ and $P^{-1/6}$, are understood to only appear in the polynomial in powers that are multiples of three.
\end{Lemma}
As its proof is somewhat short, we give it here.  Before doing so, however, we make two observations.  First, the above is essentially obvious when $\bar a(t) = \Id$.
Second (setting $\tilde t=0$ for ease), using \Cref{l.w09171} and the kinetic scaling, in which we think of $v \sim \sqrt t$ and $x\sim t^{3/2}$, every input in the polynomial $\cP_{\alpha,\beta}$ is $\sim t^{-1/2}$, making the entire polynomial $\sim t^{-|\alpha|-3|\beta|}$.  This is precisely the reason that \Cref{l.w09031} holds.

\begin{proof}
	Our proof proceeds by induction, first on the magnitude of $\alpha$ and then on the magnitude of $\beta$.  The case $|\alpha| = |\beta| = 0$ is obvious.
	
	We now consider the case where $\partial_v^\alpha = \partial_{v_i} \partial_v^{\tilde \alpha}$ for some $i\in\{1,\dots,d\}$ and $|\tilde \alpha|\geq 0$, and we assume that~\eqref{e.c9164} holds for $\partial_v^{\tilde \alpha} \Gamma$.  The derivative $\partial_{v_i}$ can, by the product rule, either fall on $\cP_{\tilde \alpha}$ or $\Gamma$.  We consider each case in turn.

	First, consider the former case; that is, consider the term that arises when the $\partial_{v_i}$ falls on $\cP_{\tilde \alpha}$. Observe that $\partial_{v_i} \cP_{\tilde \alpha}$ yields a linear combination of terms that are a $|\tilde\alpha|-1$ homogeneous polynomial multiplied by either
	\be\label{e.c032902}
		(A_0^{-1})_{ji}
			= (A_0^{-1/2})_{jk}(A_0^{-1/2})_{ki}
	\ee
	 or
	 \be\label{e.c032903}
	 	(M^T P^{-1} M)_{ji}
			= (M^TP^{-1/2})_{jk} (M^TP^{-1/2})_{ik}.
	\ee
	Each of~\eqref{e.c032902} and~\eqref{e.c032903} are 2-homogeneous in the variables of $\cP_{\alpha}$, making the resulting terms $|\tilde \alpha| - 1 + 2 = |\tilde \alpha| + 1$ homogeneous polynomials, as desired.
	
	We now consider the latter case; that is, when $\partial_{v_i}$ falls on $\Gamma$.  The conclusion is then clear as
\be\notag
	\cP_{\tilde\alpha} \partial_{v_i} \Gamma
	= \cP_{\tilde\alpha}
		\left(
			-\frac{1}{2} (A_0^{-1} v)_i 
			+ 2(M^TPx)_i
			- 2(M^TPM v)_i
		\right)
		\Gamma.
\ee
	Hence, we are finished with the proof when $|\beta| = 0$.
	
	The proof of the induction on $\beta$ is essentially the same; hence, we omit it.
\end{proof}

We are now in a position to prove \Cref{l.w09031}.
\begin{proof}[Proof of \Cref{l.w09031}]
	For ease, set $\tilde t = 0$.  We discuss first the case when $(\xi_1, \xi_2, \xi_3) = 0$ and $j=0$.

	First, for notational ease, let
	\be\notag
		I
			=\int_{\R^d}\int_{\R^d} |\partial^{\alpha}_{v} \partial^\beta_x \Gamma_{\bar a}(t,x,v;0)||x|^{r}|v|^{s}\,dvdx.
	\ee
	By \Cref{l.w09171}, we have
	\begin{equation}\notag
	\begin{split}
	I
	&=
	\det({A_0 P})^{-1/2}
	\int_{\R^d}\int_{\R^d} |\cP_{\alpha,\beta}
	\Gamma_{\bar a}(t,x,v;0)||x|^r|v|^s\,dvdx
	.
	\end{split}
	\end{equation}
		Using \Cref{l.w09032}, we notice that
	\begin{equation}\notag
	\begin{aligned}
		|\cP_{\alpha,\beta}|
			\lesssim \tilde \cP_{\alpha,\beta}\left( \frac{1}{\sqrt t}, \frac{x}{t^2}, \frac{v}{t}\right)
	\end{aligned}
	\end{equation}
	for some positive $|\alpha| + 3|\beta|$-homogeneous polynomial $\tilde \cP_{\alpha,\beta}$. 
	Therefore,  we get
\begin{equation}\notag
\begin{split}
	I
		&\lesssim
			\frac{1}{\sqrt{\det({A_0 P})}}
			\int_{\R^d}\int_{\R^d} \widetilde {\mathcal P}_{\alpha,\beta}\left(\frac{1}{\sqrt t}, \frac{x}{t^2}, \frac{v}{t}\right)			e^{-\frac{v^{T}A_0^{-1}v}{4}}
			e^{-(x-Mv)\cdot P^{-1}(x-Mv)}|x|^{r}|v|^{s}\,dvdx
		\\&=
			\frac{t^{\frac{3r}{2}+\frac{s}{2}}}{\sqrt{\det({A_0P})}}
			\int_{\R^d}\int_{\R^d} \widetilde {\mathcal P}_{\alpha,\beta}\left(\frac{1}{\sqrt t}, \frac{x}{t^2}, \frac{v}{t}\right)
			e^{-\frac{v\cdot A_0^{-1}v}{4}}
			e^{-(x-Mv)\cdot P^{-1}(x-Mv)} \frac{|x|^r}{t^\frac{3r}{2}} \frac{|v|^s}{t^\frac{s}{2}}\,dvdx.
	\end{split}
	\end{equation}
	Next, we change variables to find
	\be
		I\lesssim
			\frac{t^{\frac{3r}{2}+\frac{s}{2}-2d}}{\sqrt{\det({A_0P})}}
			\int_{\R^d}\int_{\R^d} \widetilde {\mathcal P}_{\alpha,\beta}\left(\frac{1}{\sqrt t}, \frac{\bar x}{\sqrt t}, \frac{\bar v}{\sqrt{t}}\right)
			e^{-t\frac{\bar v\cdot A_0^{-1}\bar v}{4}}
			e^{-t^3 (\bar x-t^{-1}M\bar v)\cdot P^{-1}(\bar x-t^{-1}M\bar v)} |\bar x|^r |\bar v|^s \, d\bar v d\bar x.
	\ee
	Notice that
	\be
		\widetilde {\mathcal P}_{\alpha,\beta}\left(\frac{1}{\sqrt t}, \frac{\bar x}{\sqrt t}, \frac{\bar v}{\sqrt{t}}\right)
			= t^{-\frac{|\alpha|+3|\beta|}{2}} 
				\widetilde {\mathcal P}_{\alpha,\beta}\left(1, \bar x, \bar v\right)
	\ee
	due to the homogeneity of $\widetilde{\mathcal{P}}$.  
	Hence,
		\be\notag
		I\lesssim
			\frac{t^{\frac{3r + s - |\alpha| - 3|\beta|}{2}-2d}}{\sqrt{\det({A_0 P})}}
			\int_{\R^d}\int_{\R^d} 
			\widetilde {\mathcal P}_{\alpha,\beta}\left(1,\bar v, \bar x\right)
			e^{-t\frac{\bar v\cdot A_0^{-1}\bar v}{4}}
			e^{-t^3 (\bar x-t^{-1}M\bar v)\cdot P^{-1}(\bar x-t^{-1}M\bar v)} |\bar x|^r |\bar v|^s \, d\bar v d\bar x.
	\ee
	
	We change variables one final time with $\bar y = \bar x - t^{-1} M \bar v$ to find
	\be
		\begin{split}
		I
			&\lesssim 
			\frac{t^{\frac{3r + s - |\alpha| - 3|\beta|}{2}-2d}}{\sqrt{\det({A_0 P})}}
			\int_{\R^d}\int_{\R^d} 
			\widetilde {\mathcal P}_{\alpha,\beta}\left(1,\bar v, \bar x\right)
			e^{-t\frac{\bar v\cdot A_0^{-1}\bar v}{4}}
			e^{-t^3 \bar y\cdot P^{-1}\bar y} |\bar y + t^{-1} M \bar v|^r |\bar v|^s \, d\bar v d\bar x
			\\&\lesssim
			\frac{t^{\frac{3r + s - |\alpha| - 3|\beta|}{2}-2d}}{\sqrt{\det({A_0 P})}}
			\int_{\R^d}\int_{\R^d} 
			\widetilde {\mathcal P}_{\alpha,\beta}\left(1,\bar v, \bar x\right)
			e^{-t\frac{\bar v\cdot A_0^{-1}\bar v}{4}}
			e^{-t^3 \bar y\cdot P^{-1}\bar y} (|\bar y|^r + |\bar v|^r) |\bar v|^s \, d\bar v d\bar x.
		\end{split}
	\ee
	In the last inequality we used \Cref{l.w09032} to bound $t^{-1} |M| \lesssim 1$.
	
	At this point, it follows from \Cref{l.w09032} that the quadratic terms in the exponential are bounded below as 
	\be\notag
		t \bar v\cdot A_0^{-1} \bar v
			+ t^3 \bar y\cdot P^{-1} \bar y
			\gtrsim |\bar v|^2 + |\bar y|^2.
	\ee
	The conclusion follows then from a simple calculation:
	\be\notag
		I
			\lesssim \frac{t^{\frac{3r + s - |\alpha|-3|\beta|}{2}-2d}}{\sqrt{\det({A_0P})}}.
	\ee
	The proof of this case is concluded after applying \Cref{l.w09032} again in order to bound the determinant.
	
	The case where $j>0$ reduces to the case above via the identity:
	\be\notag
		\partial_t \Gamma_{\bar a}
		= v \cdot \nabla_x \Gamma_{\bar a}
		+ \tr(a(t) D^2_v \Gamma_{\bar a}).
		\ee
	This concludes the proof of all cases where $(\xi_1,\xi_2,\xi_3) = 0$.

The general case can easily be handled as follows.  First change variables:
	\be\notag
		\begin{split}
		\int_{\R^d}\int_{\R^d}
			&\max_{(0,\xi_1,\xi_2) \in Q_{t/2}}|\partial^{j}_{t}\partial_{x}^{\beta}\partial_{v}^{\alpha}\Gamma_{\bar a}(t,x+\xi_1,v+\xi_2;0)|x|^{r}|v|^{s}\,dxdv
			\\&	= \int_{\R^d}\int_{\R^d}
			\max_{(0,\xi_1,\xi_2) \in Q_{t/2}}|\partial^{j}_{t}\partial_{x}^{\beta}\partial_{v}^{\alpha}\Gamma_{\bar a}(t,x,v;0)|x- \xi_1|^{r}|v-\xi_2|^{s}\,dxdv.
		\end{split}
	\ee
	Next, using the inequalities
	\be\notag
		|x- \xi_1|^{r}
			\lesssim |x|^r + |\xi_1|^r
		\quad\text{ and }\quad
		|v-\xi_2|^{s}
			\lesssim |v|^s + |\xi_2|^s.
	\ee
	At this point, the four resulting integrals may be estimated using the case above (keeping in mind the conditions $|\xi_1|\leq t^{3/2}$ and $|\xi_2| \leq \sqrt t$).  This concludes the proof.
\end{proof}

\subsubsection{The proof of \Cref{l.w09032}: understanding the matrices $A_i$, $P$, and $M$}\label{s.matrices}

\begin{proof}
	We note that the upper bounds in all cases (i), (ii), and (iii) are obvious from the assumptions on $\bar a$~\eqref{e.homogeneous_assumptions} and the definition of the matrices~\eqref{e.A_i}.  The lower bounds of $A_i$ in (ii) are also obvious.  Hence, we need only prove the lower bound for $P$ in (i).

	For ease, we set $\tilde t = 0$ and simply drop the ``$;0$'' notation from all quantities.
	
	To obtain this lower bound, notice that it suffices to establish a uniform bound of the form
	\be\label{e.c033009}
		w \cdot (P(t)w)
			\gtrsim t^3
	\ee
	for any vector $w\in \R^d$ with $|w| = 1$.  We proceed by analyzing the time derivative of $P$.  First,
	\begin{equation}\notag
	\begin{split}
	P'(t)
		&= (A_2 -A_1 A_0^{-1} A_1)'(t)
	\\&
	=
	t^2 \bar a-t\bar aA_0^{-1}A_1 +A_1 (A_0^{-1} \bar aA_0^{-1} ) A_1 -A_1 A_0^{-1}t\bar a
	\\&=
	(t\sqrt{\bar a}-A_1 A^{-1}\sqrt{\bar a})(t\sqrt{\bar a}-\sqrt{\bar a}A^{-1}A_1)
	= M^T \bar a M 
	\geq 0
	.
	\end{split}
	\end{equation}
	Thus,
	\be\label{e.c033007}
		w\cdot P(t)w
			= \int_0^t (M(s) w) \cdot \bar a M(s) w \, ds
			\geq \frac{1}{\Lambda}
				\int_0^t |M(s) w|^2 \, ds.
	\ee
	In order to establish~\eqref{e.c033009}, it is enough to show
	\be\label{e.c033006}
		\int_0^t |M(s) w|^2\, ds
			\gtrsim t^3.
	\ee
	This is our focus for the remainder of the proof.
	

	To obtain this lower bound, we use the following intuition. 
	Recall in equation \eqref{e.A_i}
		 \begin{equation}\notag
		\begin{split}
		&A_i(s)=\int_{0}^{s} \tau^i \bar a(\tau)\,d\tau,
		\quad\text{ and }\quad	
		M(s)
		= s\id - A_0^{-1}(s) A_1(s).
		\end{split}
		\end{equation} 
	The time derivative of $M$ is
	\be\label{e.c033005}
	\begin{split}
		M'(s)
			&= \id - (A_0^{-1} A_1)'(s)
			= \id + A_0(s)^{-1} \bar a(s) A_0^{-1}(s) A_1(s) - s A_0(s)^{-1} \bar a(s)
			\\&
			= \id - A_0(s)^{-1} \bar a(s)M(s).
	\end{split}
	\ee
	From~\eqref{e.c033005}, we see that when $M(s) w$ is `small,' that is $o(t)$, $M'(s) w$ is approximately $w$.  That is $M(s)w$ moves radially (in the direction of $w$) with a velocity $\approx 1$ away from the origin.  This means that, eventually, $M(s) w$ will move radially across a distance $O(t)$ at a bounded velocity.  This would yield the desired bound~\eqref{e.c033006}. 

	In order to make this rigorous, we proceed in three steps.  Fix $\eps >0$ sufficiently small in a way to be determined.  The first step is to note that either there is an interval $[\eps t, 2\eps t]$ where $M(s) w$ always has magnitude $O(t)$ or not.  If so, we are done.  If not, we proceed to the second step.  The second step takes a time $t_0$ in the interval $[\eps t, 2\eps t]$ in which $M(t_0) w$ is `small' and shows that it gets `big.'  The third step is to show that $M(t_0) w$ remains `big'.  The second and third steps are dependent on the time derivative of $M$.

%
	
	\smallskip
	
	{\bf Step one:} Notice that $|M(0) w| = 0$.  If
	\be
		|M(s) w| \geq \epsilon^4 t
			\qquad \text{ for all } s \in [\eps t, 2\eps t]
	\ee
	then we are finished with the proof.  Hence, assume that there is
	\be
		t_0 \in [\eps t, 2\eps t]
			\quad\text{ such that } |M(t_0) w| < \eps^3 t.
	\ee
	
	\smallskip
	
	{\bf Step two:} We claim that
	\be\label{e.c033001}
		t_1 := \inf \{ s > t_0 : |M(s) w| \geq \eps^3 t\} \leq 3 \eps t.
	\ee
	Roughly, $t_1$ is the first time after $t_0$ that $|M(s)w|$ becomes `big,' that is, has norm $\eps^3 t$.
	
	Using the time derivative of $M$~\eqref{e.c033005}, we obtain the identity
	\be\label{e.c033003}
		M(t_1) w
			= M(t_0) w
				+ \int_{t_0}^{t_1}\left(w - A_0(s)^{-1} a(s) M(s) w\right) \, ds.
	\ee
	
	Combining~\eqref{e.c033003} with the definition of $t_1$~\eqref{e.c033001} and the bound \Cref{l.w09032}.(ii), we find
	\be\label{e.c033004}
	\begin{split}
		t_1 - t_0
			&= (t_1 - t_0)|w|
			= \Big|\int_{t_0}^{t_1} w \,ds\Big|
			= \Big|M(t_1) w - M(t_0)w + \int_{t_0}^{t_1} A_0(s)^{-1} \bar a(s) M(s) w \,ds \Big|
			\\&
			\leq |M(t_1) w|
				+ |M(t_0) w|
				+ \int_{t_0}^{t_1} |A_0(s)^{-1} \bar a(s) M(s) w|\, ds
			\\&
			\leq 2 \eps^3 t
				+ \int_{t_0}^{t_1} \frac{C |M(s) w|}{s^2}  \, ds
			\leq 2 \eps^3 t
				+ \int_{t_0}^{t_1} \frac{C \eps^3 t}{s}  \, ds
			\leq 2 \eps^3 t + \frac{C \eps^3 t (t_1 - t_0)}{\eps t}.
	\end{split}
	\ee
	where $C$ is a universal constant depending only on $d$ and $\Lambda$.  The last inequality uses that $s \geq t_0 \geq \eps t$.  Before continuing, we note that the last integral in~\eqref{e.c033004} reveals the necessity of Step One, above.  Indeed, the final integral above is not bounded for $t_0$ near $0$. 
	Step One allows us to avoid this singularity. 
	
	Returning to~\eqref{e.c033004}, notice that, if $\eps$ is sufficiently small then $C\eps^2 < 1/2$.  Thus, after rearranging~\eqref{e.c033004}, we find
	\be\notag
		\frac{t_1-t_0}{2}
			\leq 2 \eps^3 t.
	\ee
	Rearranging this, recalling that $t_0 \leq 2 \eps t$, and further decreasing $\eps$, we obtain
	\be\notag
		t_1
			\leq t_0 + 4\eps^2 t
			< 2 \eps t + \eps t.
	\ee
	Hence~\eqref{e.c033001} is established.

	\smallskip
	
	{\bf Step Three:} We claim that
	\be\label{e.w05181}
		t_2
			:= \sup \{ s \in (t_1, t] : |M(s) w| \geq \eps^4 t\}
			\geq t_1 + \eps^4 t.
	\ee
	Roughly, $t_2$ is the first time after $t_1$ (at which time $|M(s) w|$ is `big') that $|M(s) w|$ becomes `small,' that is $\eps^4 t$.
	
	Before showing this, we claim this allows us to conclude.  Indeed, 
	\be\notag
		|M(s) w| \geq \eps^4 t
			\quad \text{ for } s \in (t_1,t_2)
			\qquad\text{ and }\qquad
		t_2 - t_1 \geq \eps^4 t.
	\ee
	Hence, 
	\be\notag
		\int_0^t |M(s) w|^2\, ds
			\geq	\int_{t_1}^{t_2} |M(s) w|^2\, ds
			\geq (t_2 - t_1) (\eps^4 t)^2
			\geq \epsilon^{12} t^3.
	\ee
	In view of~\eqref{e.c033007}, this establishes the claim~\eqref{e.c033006}, which concludes the proof.  Thus, it is enough to prove~\eqref{e.w05181}, which is our focus now.
	
	If $t_2 = t$, we are finished.  Hence, we assume that $t_2 < t$, which implies that
	\be\label{e.c033010}
		|M(t_2) w| = \eps^4 t.
	\ee
	Also, using~\eqref{e.c033005} once again, we find
	\be\notag
		M(t_1) w
			- M(t_2)w
			= -(t_2-t_1)w + \int_{t_1}^{t_2} A_0^{-1}(s) \bar a(s) M(s) w \, ds.
	\ee
	Combining the two identities above, and recalling from~\eqref{e.c033001} that $|M(t_1)w| = \eps^3 t$, we find
	\be\notag
		\eps^3 t(1 - \eps)
			\leq |M(t_1) w - M(t_2) w|
			\leq (t_2-t_1) 
				+ \int_{t_1}^{t_2} C \, ds
			= (C+1) (t_2- t_1).
	\ee
	Rearranging this and decreasing $\eps$ if necessary, we find~\eqref{e.w05181}.  This concludes the proof.
\end{proof}

\subsection{The second step: full Schauder estimates by perturbing off the homogeneous problem}\label{s.w05171}

By a careful procedure taking into account the natural scalings and available interpolations, we can perturb off the $(x,v)$-homogeneous problem in order to obtain the full Schauder estimates.  In short, we finish the proof of \Cref{t.Schauder} by leveraging \Cref{l.w12041}.

We begin by stating two important technical lemmas.  The proof of the first is given in \cite{LH} and the second is standard (it can be seen easily by scaling, for example), but a proof can be found in~\cite[Lemma~2.10]{imbert2018schauder}.

\begin{Lemma}[Lemma~4.3 in~\cite{LH}]\label{l.w12042}
	Let $\omega(r)>0$ be bounded in $[r_0,r_1]$ with $r_0 \geq 0$. Suppose that there is $\mu \in (0,1)$ and constants $A$, $B$, $p\geq 0$ so that, for all $r_0\leq r<R\leq r_1$,
	\begin{equation}\label{e.w12044}
	\begin{split}
		\omega(r)\leq \mu \omega(R)+\frac{A}{(R-r)^p}+B.
	\end{split}
	\end{equation}
	Then for any $r_0\leq r<R\leq r_1$, there holds
	\begin{equation}\label{e.w12045}
	\begin{split}
	\omega(r)\lesssim \frac{A}{(R-r)^p}+B
	,
	\end{split}
	\end{equation}
	where the implied constant depends only on $\mu$ and $p$.
\end{Lemma}

\begin{Lemma}[Interpolation inequalities]\label{l.w12022}
Fix any $Q=Q_{r}$ for any $r \geq 1/2$ and any $\alpha \in (0,1)$.  For any $\epsilon>0$, the following hold: 
	\be
		\begin{split}
		[u]_{\calpha(Q)}
			&\lesssim
				\epsilon^{2}\left([u]_{C^{(2+\alpha)/3}_x(Q)} + [D^{2}_{v}u]_{C^{\alpha/3}_xC_v^\alpha(Q)}\right)
				+
				\epsilon^{-\alpha}\|u\|_{L^{\infty}(Q)},
	\\
		[D_{v}u]_{\calpha(Q)}
			&\lesssim
				\epsilon^{}\left([u]_{C^{(2+\alpha)/3}_x(Q)} + [D^{2}_{v}u]_{C^{\alpha/3}_xC_v^\alpha(Q)}\right)
				+
				\epsilon^{-\alpha-1}\|u\|_{L^{\infty}(Q)},
	\\
		\|D_v u\|_{L^{\infty}(Q)}
			&\lesssim
		\epsilon^{\alpha+1}[D^{2}_{v}u]_{C^{\alpha/3}_xC_v^\alpha(Q)}
		+
		\epsilon^{-1}\|u\|_{L^{\infty}(Q)},
		\qquad\text{ and}
	\\
	\|D_v^2 u\|_{L^{\infty}(Q)}
		&\lesssim
			\epsilon^{\alpha}[D^{2}_{v}u]_{C^{\alpha/3}_xC_v^\alpha(Q)}
			+
			\epsilon^{-2}\|u\|_{L^{\infty}(Q)}
	.
	\end{split}
	\end{equation}
\end{Lemma}

With these in hand, we now prove the full Schauder estimates.
\begin{proof}[Proof of \Cref{t.Schauder}]
%
%
	We estimate $[D_{v}^{2}u]_{\calpha(Q_{1/2})}$ and omit the proof of the other terms as they are similar.
	For succinctness, in this proof, we use the following notation:
	\begin{equation}\label{e.w03312}
		[u]'_{2+\alpha,r}
			:=[D^{2}_{v}u]_{\calpha(Q_r)}
				+[u]_{C^{(\alpha+2)/3}_x(Q_r)}.
	\end{equation}

	The key estimate that we establish is the following.  There is $\eps_0>0$ sufficiently small so that, with
	\be\label{e.theta_0}
		\theta_0 := \min\Big\{\frac{1}{8}, \eps_0 [\bar a]_{\calpha(Q_1)}^{-1/\alpha}\Big\}
	\ee
	then
	\be\label{e.c51201}
	\begin{split}
	[f]_{2+\alpha,r}'
		\leq &\frac{1}{2} [f]_{2+\alpha, r+2\theta}'
				+ C([\bar c]_{\calpha(Q_1)} + \theta^{-2-\alpha}) \|f\|_{L^\infty(Q_1)}
				+ C\theta^{-\alpha} \|g\|_{\calpha(Q_1)},
	\end{split}
	\ee
	for some $C>0$ and all $\theta \in (0,\theta_0)$ and $r \in [1/4, 3/4]$.  
	The proof of~\eqref{e.c51201} is complicated, so we postpone it until after show how \Cref{t.Schauder} follows from it.

	In order to prove \Cref{t.Schauder} from~\eqref{e.c51201}, we first rewrite~\eqref{e.c51201} in a manner more adapted to \Cref{l.w12042}.  Indeed, applying \Cref{l.w12042} with,  in its notation, the choices
	\be
		\begin{split}
		&r_1 = 1/4,
		\quad
		r_2 = 3/4,
		\quad
		\omega(r) = [u]_{2+\alpha,Q_r}',
		\quad
		R = r + 2\theta_0,
		\quad
		\mu = \frac{1}{2},
		\\&
		A = \|f\|_{L^\infty(Q_1)} + \theta_0^2 [g]_{\calpha(Q_1)},
		\quad
		B = C[\bar c]_{\alpha(Q_1)} \|f\|_{L^\infty(Q_1)},
		\quad\text{ and }\quad
		p = 2 + \alpha,
		\end{split}
	\ee
	yields
	\begin{equation}\label{e.w12052}
	\begin{split}
		&[u]'_{2+\alpha,1/2}
			\lesssim
				\theta_0^{-2-\alpha}
					\left(
						\|f\|_{L^\infty(Q_1)}
						+ \theta_0^2 [g]_{\calpha(Q_1)}
					\right)
				+ [\bar c]_{\calpha(Q_1)}\|f\|_{L^\infty(Q_1)}
			\\&\quad
			\lesssim \left(1+ [\bar c]_{\calpha(Q_1)} + [\bar a]_{\calpha(Q_1)}^{1+ \frac{2}{\alpha}}\right) \|f\|_{L^\infty(Q_1)}
				+ \left(1 + [\bar a]_{\calpha(Q_1)}\right)
				[g]_{\calpha(Q_1)}.
	\end{split}
	\end{equation}
	Thus, \Cref{t.Schauder} is proved, up to establishing~\eqref{e.c51201}.

	We now prove~\eqref{e.c51201}.  We argue under the assumption that
	\be\notag
		[D^2_v f]_{\calpha(Q_r)}
			\geq [f]_{C^{(2+\alpha)/3}_x(Q_r)}
	\ee
	so that
	\be\notag
		[f]_{2+\alpha, r}'
			\leq 2 [D^2_v f]_{\calpha(Q_r)},
	\ee
	although the proof is similar in the opposite case. 
	
	Fix $z_0,z_1 \in Q_r$ with $t_0 = t_1$ so that
	\be\notag
		\frac{|D^2_v f(z_0) - D^2_v f(z_1)|}{|x_0-x_1|^{\alpha/3} + |v_0 - v_1|^\alpha}
			\geq \frac{1}{2} [D^2_v f]_{\calpha(Q_r)}
			\geq \frac{1}{4} [f]_{2+\alpha,Q_r}'.
	\ee
	Up to a change of variables, we may assume that $z_1 = 0$, which make the expressions in the sequel simpler.

	Fix $\theta \in (0,\theta_0)$.  If
	\be\label{e.c51101}
		|x_0|^{\alpha/3} + |v_0|^\alpha
			> \theta,
	\ee
	then we have, using \Cref{l.w12022}, 
	\be\notag
		\begin{split}
		[f]_{2+\alpha,Q_r}'
			&\lesssim \frac{|D^2_v f(z_0) - D^2_v f(0)|}{|x-x'|^{\alpha/3} + |v_0|^\alpha}
			\lesssim \theta^{-\alpha} \|D^2_v f\|_{L^\infty(Q_r)}
			\\&
			\leq \frac{1}{2} [ D^2_v f]_{\calpha(Q_r)} + C\theta^{-2-\alpha} \|f\|_{L^\infty(Q_r)},
		\end{split}
	\ee
	and~\eqref{e.c51201} is proved.
	
	Next we consider the case when~\eqref{e.c51101} does not hold.   We introduce a cut-off function $0 \leq \chi\leq 1$ such that
	\be\notag
		\chi(t,x, v)
			=\begin{cases}
				1 \qquad&\text{ if } |t|^{1/2} + |x|^{1/3} + |v| \leq \theta,\\
				0 \qquad&\text{ if } |t|^{1/2} + |x|^{1/3} + |v| \geq 2\theta,
			\end{cases}
	\ee
	and
	\begin{equation}\label{e.w12031}
		[(\partial_t +v\cdot \nabla_x)\chi]_{\calpha(Q_1)}+[D_{v}^{2}\chi]_{\calpha(Q_1)}
	\lesssim \theta^{-2-\alpha}
	.
	\end{equation}
	We note that estimates on the other norms and semi-norms of $\chi$ can be obtained easily via \Cref{l.w12022}. 
	Additionally, to make the notation simpler, we define
	\be\notag
		\tilde a(t)
			= \bar a(t,0,0)
		\qquad\text{and}\qquad
		L = \partial_t + v\cdot\nabla_x - \tr(\bar aD^2_v\cdot).
	\ee

	By \Cref{l.w12041}, we have
	\begin{equation}\label{e.w12041}
	\begin{split}
	[D_{v}^{2}f]_{\calpha(Q_r)}
	&\leq
	[D^2_v(\chi f)]_{\calpha(Q_1)} 
	\lesssim [\partial_t(\chi f) + v\cdot\nabla_x(\chi f)
		- \tr(\tilde a D^2_v (\chi f))]_{\calpha(Q_1)}
	\\&
	\lesssim
	[L(\chi f)]_{\calpha(Q_1)}
	+
	[\tr((\bar a - \tilde a)D^2_v(\chi f))]_{\calpha(Q_1)}.
	\end{split}
	\end{equation}
	
	We consider the first term on the right hand side of~\eqref{e.w12041}.  Using the equation for $u$, we see
	\begin{equation}\notag
		L(\chi f)
			=
			f L \chi
			- 2\bar a\nabla_{v}f\nabla_{v}\chi
			+ \bar c \chi f
			+\chi g.
	\end{equation}
	Thus, by the triangle inequality
	\begin{equation}\notag
	\begin{split}
	&[L (\chi f)
		]_{\calpha(Q_1)}
	\\&\leq
	[\chi g]_{\calpha(Q_1)} + [fL\chi]_{\calpha(Q_1)} 
	+[\bar c\chi f]_{\calpha(Q_1)}
	+
	2[\bar a\nabla_{v}f\nabla_{v}\chi]_{\calpha(Q_1)}
	\\&=
	I_1 +I_2 +I_3+I_4.
	\end{split}
	\end{equation}
	
	For $I_1$, we use the boundedness of the cut-off function and have
	\begin{equation}\notag
	\begin{split}
	I_1
	&\lesssim
		\theta^{-\alpha}\|g\|_{\calpha(Q_1)}.
	\end{split}
	\end{equation}
	
	Next we consider $I_2$. Keeping in mind the support of $\chi$ and using the interpolation inequality \Cref{l.w12022} and \eqref{e.w12031} yields
	\be\notag
		\begin{split}
		&I_2
		\lesssim
		[f]_{\calpha(Q_{r+2\theta})} 
		\|L\chi\|_{L^{\infty}(Q_{r+2\theta})} 
		+
		[L\chi]_{\calpha(Q_{r+2\theta})} 
		\|f\|_{L^{\infty}(Q_{r+2\theta})} 
		\\&\lesssim
		\theta^{-2}
		[f]_{\calpha(Q_{r+2\theta})}
		+
		(\theta^{-2-\alpha}
		+ [\bar a D^2_v \chi]_{\calpha(Q_{r+2\theta})})
		\|f\|_{L^{\infty}(Q_{r+2\theta})}.
	\end{split}
	\ee
	Using again \Cref{l.w12022} and~\eqref{e.w12031}, we have
	\be\notag
	\begin{split}
		[\bar a D^2_v \chi]_{\calpha(Q_{r+2\theta})}
			&\leq 
				[\bar a]_{\calpha(Q_{r+2\theta})} \| D^2_v \chi\|_{L^\infty(Q_{r+2\theta})}
				+  \|\bar a\|_{L^\infty(Q_{r+2\theta})} [ D^2_v \chi]_{\calpha(Q_{r+2\theta})}
			\\&\lesssim
				\theta^{-2} [\bar a]_{\calpha(Q_{r+2\theta})}
				+ \theta^{-2-\alpha}. 
	\end{split}
	\ee
	and, for $\eps>0$ to be chosen depending only on $d$, $\alpha$, and $\Lambda$ (recall~\eqref{e.ellipticity}),
	\be\notag
		\theta^{-2} [f]_{\calpha(Q_{r+2\theta})}
			\lesssim \eps [f]_{2+\alpha, Q_{r+2\theta}}' + \theta^{-2 - \alpha}\|f\|_{L^\infty(Q_1)}.
	\ee
	Note that, as $\eps$ will not be chosen to depend on $[\bar a]_{\calpha(Q_1)}$ or $\theta$, we omit all negative powers of $\eps$.

	Therefore, we conclude that
	\be\notag
		I_2
			\lesssim \eps[f]_{2+\alpha,Q_{r+2\theta}}' + (\theta^{-2} [\bar a]_{\calpha(Q_{r+2\theta})} + \theta^{-2 - \alpha})\|f\|_{L^\infty(Q_1)}.
	\ee
	
	The terms $I_3$ and $I_4$ may be handled similarly to obtain
	\be\notag
		I_3
			\lesssim ([\bar c]_{\calpha(Q_1)} + \theta^{-\alpha}) \|f\|_{L^\infty(Q_1)}
				+ \eps [f]'_{2+\alpha,Q_{r+2\theta}}
	\ee
	and
	\be\notag
		\begin{split}
		I_4
			&\lesssim ([a]_{\calpha(Q_1)}  \theta^{-1}
				+ 
				\theta^{-1-\alpha}
				)
				\|\nabla_v f\|_{L^\infty(Q_{r+2\theta})}
				+ \theta^{-1} [\nabla_v f]_{\calpha(Q_{r+2\theta})}
			\\&
			\lesssim ([\bar a]_{\calpha(Q_1)}  \theta^{-1}
				+ 
				\theta^{-1-\alpha}
				)
				(\eps \theta^{1+\alpha}[D^2_vf]_{\calpha(Q_{r+2\theta})} + 
				\theta^{-1}\|f\|_{L^\infty(Q_{r+2\theta})})
				\\&\qquad
				+ \theta^{-1} (
					\eps \theta [D^2_v f]_{\calpha(Q_{r+2\theta})}
					+
					\theta^{-1-\alpha}\|f\|_{\calpha(Q_{r+2\theta})}
				)
			\\&\lesssim
				(\eps + \theta^\alpha [\bar a]_{\calpha(Q_1)}) [D^2_v f]_{\calpha(Q_{r+2\theta})}
				+ 
				(\theta^{-2} [\bar a]_{\calpha(Q_1)}
					+ \theta^{-2-\alpha}) \|f\|_{L^\infty(Q_1)}.
		\end{split}
	\ee
	Hence, we conclude that
	\be\label{e.c51102}
		\begin{split}
		&[L(\chi f)]_{\calpha(Q_1)}
			\lesssim (\eps + \theta^\alpha [\bar a]_{\calpha(Q_1)}) [f]_{2+\alpha,Q_{r+2\theta}}'
				\\&\qquad
				+ ([c]_{\calpha(Q_1)} + \theta^{-2} [\bar a]_{\calpha(Q_1)}+ \theta^{-2-\alpha}) \|f\|_{L^\infty(Q_1)}
				+ \theta^{-\alpha} \|g\|_{\calpha(Q_1)}.
		\end{split}
	\ee

We now consider the second term on the right hand side of~\eqref{e.w12041}.  We begin with the usual splitting:
\be\label{e.c51204}
	\begin{split}
	[\tr(\bar a-\tilde a) D^2_v(\chi f)]_{\calpha(Q_1)}
		\lesssim &\|\bar a-\tilde a\|_{L^\infty(\supp(\chi))} [D^2_v(\chi f)]_{\calpha(Q_1)}
			\\&
			+ [\bar a-\tilde a]_{\calpha(Q_1)} \|D^2_v(\chi f)\|_{L^\infty(Q_1)}.
	\end{split}
\ee
The second term in~\eqref{e.c51204} can be handled easily using the methods above (recall \Cref{l.w12022} and \eqref{e.w12031}):
\be\notag
	[\tr(\bar a-\tilde a)]_{\calpha(Q_1)} \|D^2_v(\chi f)\|_{L^\infty(Q_1)}
		\lesssim [\bar a]_{\calpha(Q_1)}
			\left(
				\theta^\alpha [f]_{2+\alpha, Q_{r+2\theta}}'
				+ \theta^{-2} \|f\|_{L^\infty(Q_1)}
			\right).
\ee
Estimating the first term in~\eqref{e.c51204} uses the fact that $\chi$ has support of size $\theta$:
\be\notag
	\|\bar a-\tilde a\|_{L^\infty(\supp(\chi))}
		\lesssim \theta^\alpha [\bar a]_{\calpha(Q_1)}.
\ee
After applying \Cref{l.w12022} and \eqref{e.w12031}, we arrive at
\be\notag
	\begin{split}
	\|\bar a-\tilde a\|_{L^\infty(\supp(\chi))} &[D^2_v(\chi f)]_{\calpha(Q_1)}
		\lesssim \theta^\alpha [\bar a]_{\calpha(Q_1)}
			\left(
				[f]_{2+\alpha,Q_{r+2\theta}}'
				+ \theta^{-2} \|f\|_{L^\infty(Q_1)}
			\right).
	\end{split}
\ee
Therefore, we obtain the following bound on the second term on the right hand side of~\eqref{e.w12041}:
\be\notag
	[\tr(\bar a-\tilde a) D^2_v(\chi f)]_{\calpha(Q_1)}
		\lesssim \theta^\alpha [\bar a]_{\calpha(Q_1)}
			\left(
				[f]_{2+\alpha,Q_{r+2\theta}}'
				+ \theta^{-2} \|f\|_{L^\infty(Q_1)}
			\right).
\ee

Combining all above estimates, we have
\be\label{e.c51205}
	\begin{split}
		&[f]_{2+\alpha, r}'
			\leq C(\eps + \theta^\alpha [\bar a]_{\calpha(Q_1)}) [f]_{2+\alpha,Q_{r+2\theta}}'
				\\&\qquad
				+ C([\bar c]_{\calpha(Q_1)} + \theta^{-2} [\bar a]_{\calpha(Q_1)}+ \theta^{-2-\alpha}) \|f\|_{L^\infty(Q_1)}
				+ C\theta^{-\alpha} \|g\|_{\calpha(Q_1)},
		\end{split}
\ee
where $C$ is some universal constant.
Choosing $\eps$ and $\eps_0$ sufficiently small and recalling the definition of $\theta_0$~\eqref{e.theta_0} and that $\theta < \theta_0$, we have
\be\notag
	C(\eps + \theta^\alpha)[\bar a]_{\calpha(Q_1)}
		\leq \frac{1}{2}
	\quad\text{ and }\quad
	\theta^{-2} [\bar a]_{\calpha(Q_1)}
		\lesssim \theta^{-2-\alpha}.
\ee
Using this in~\eqref{e.c51205}, we obtain~\eqref{e.c51201}, which concludes the proof.
\end{proof}


\section{Uniqueness for the Landau equation: \Cref{t.Landau}} \label{s.Landau}

Before beginning the proof, we review a few useful bounds that follow from our assumptions.  For any $h \in L^{\infty,5+\gamma + \eta}$ for any $\eta>0$, we have
\be\label{e.bar_a_above}
	e \cdot \bar a^h(t,x,v) e
		\lesssim
		\|h\|_{L^{\infty,5+\gamma + \eta}}
		\begin{cases}
			\vv^\gamma
				\qquad &\text{ if } e \parallel v,\\
			\vv^{2+\gamma}
				\qquad &\text{ otherwise,}
		\end{cases}
\ee
for any $e \in \mathbb{S}^2$.  The lower order coefficient $\bar c^h$ satisfies a similar bound
\be\label{e.bar_c}
	0 \leq \bar c^h(t,x,v)
		\lesssim \vv^\gamma\|h\|_{L^{\infty,3}}.
\ee
These bounds are not optimal in the weight; it is clear that~$3$ can be replaced by any $k > 3+\gamma$.  The proofs of~\eqref{e.bar_a_above} and~\eqref{e.bar_c} are straightforward but can be seen in \cite[Lemma~2.1]{HST2019rough}.

Finally, due to the assumption~\eqref{e.nondegeneracy}, the solution $f$ that is the subject of \Cref{t.Landau} satisfies a matching lower bound to~\eqref{e.bar_a_above}:
\be\label{e.bar_a_below}
	e \cdot \bar a^f(t,x,v) e
		\gtrsim
		\begin{cases}
			\vv^{2+\gamma}
				\qquad &\text{ if } e \perp v,\\
			\vv^\gamma
				\qquad &\text{ otherwise.}
		\end{cases}
\ee
Here we suppress the explicit dependence on $f$ as it depends in a complicated way on $\delta$, $r$, $R$, and $\|f\|_{L^{\infty,k}}$.  This inequality~\eqref{e.bar_a_below} follows from \cite[Lemma~2.5]{HST2019rough}.

We now state the main quantitative estimate that allows us to deduce uniqueness (\Cref{t.Landau}), which is postponed until \Cref{s.Hessian}.  This lemma requires \Cref{t.Schauder} in a crucial way.

\begin{Proposition}\label{l.Hessian}
	Under the assumptions of \Cref{t.Landau}, there are $\alpha' \in (0,\alpha)$ and $\theta'\in(0,\theta)$ so that
	\be
		\frac{\theta'}{2} \frac{\alpha'}{2+\alpha'}
			> 1,
	\ee
	and $T_0 < 1/2$ such that, for any $t\in [0,T_0]$,
	\be
		\|\vv^7 D^2_v f(t)\|_{L^\infty(\R^6)}
			\lesssim \frac{1}{t \log(\frac{1}{t})^{\frac{\theta'}{2} \frac{\alpha'}{2+\alpha'}}}.
	\ee
	The final time $T_0$ depends only on $\alpha$, $\theta$, $k$, $\|f_0\|_{L^{\infty,k}}$, and $\|f_0\|_{\logalpha}$.
\end{Proposition}


With this in hand, we are in position to prove the second main theorem, the uniqueness of solutions to the Landau equation with initial data having H\"older regularity in $x$ and $\log$-H\"older regularity in $v$.  While the estimate of \Cref{l.Hessian} is different from its analogue in~\cite[Lemma 4.3 and Proposition 4.4]{HST2019rough}, its application in deducing uniqueness is quite similar to the proof of uniqueness in~\cite{HST2019rough}.  However, we provide the proof as some technical details must be altered.

\begin{proof}[Proof of \Cref{t.Landau}]
	For succinctness, we set
	\be
		\ell = 5 + \gamma + \eta,
	\ee
	and, without loss of generality, we may assume that
	\be\label{e.c52511}
		\ell \leq 5.
	\ee
	
	Let $r \in C(0,T_0] \cap L^1[0, T_0]$ be a positive function to be determined and define
	\begin{equation}
	w=e^{-\int_{0}^{t}r(s)\,ds}(g-f)
	\quad\text{ and }\quad
	W = \vv^{2\ell} w^2.
	\end{equation}
	
	Our goal is to show that $W \equiv 0$ as this immediately implies that $f\equiv g$.   We proceed by contradiction, assuming that there is $\epsilon>0$ such that
	\begin{equation}\label{e.c033011}
	\sup_{[0,T_0]\times \R^6} W(t,x,v)> \eps.
	\end{equation}
	Following the work in \cite[Proposition 5.2]{HST2019rough}, we may find a point $z_\eps = (t_\eps, x_\eps, z_\eps)$ with $t_\eps>0$ such that
	\be\label{e.W_max}
	W(z_\eps)
	= \eps
	\quad\text{ and }\quad
	\sup_{[0,t_\eps]\times \R^6} W(t,x,v) \leq \eps.
	\ee

	%
	%
	%
	%
	%
	%
	
	Next, a direct computation from equation \eqref{e.Landau} yields 
\begin{equation}\label{e.w03012}
\begin{split}
	\partial_t W+ v\cdot \nabla_x W
		= &\tr{(\bar a^g D_{v}^{2}W)}
			- \frac{1}{2} W^{-1} \nabla_v W \cdot (\bar a^g \nabla_v W)
			+ 2\ell \vv^{-2} v \cdot (\bar a^g \nabla_v W)
		\\&+
			\big[ 2\ell(\ell+2) \vv^{-4}v\cdot (\bar a^g v)
				-2\ell \vv^{-2}\tr{(\bar a^g)}+ 2\bar c^g \big]W
		\\&
			+ 2 \vv^{2\ell} w\tr{(\bar a^w D_{v}^{2}f)}
			+ 2\vv^{2\ell} w \bar c^w f
			-2rW
.
\end{split}
\end{equation}
It is in this step that we need the technical condition $W(z_\eps) > 0$; indeed, otherwise the second term on the right hand side would not be well-defined.


	We notice three things.  First, as $z_\epsilon$ is a maximum point (recall~\eqref{e.W_max}), it follows that, at $z_\eps$,
	\be\label{e.c52409}
	\nabla_v W=0,
	\quad
	D_{v}^{2}W\leq 0,
	\quad\text{ and }\quad
	(\partial_t +v\cdot \nabla_x )W\geq 0.
	\ee
	Second, we have that,
	\be\label{e.c033012}
	\|w\|_{L^{\infty,\ell}([0,t_\epsilon]\times \R^6)}
	= W(z_\epsilon).
	\ee
	%
	%
	At this point, we drop the indication of the domain from the $L^\infty$-norms as it will always be $[0,t_\epsilon]\times \R^6$.
	
	Next, after using~\eqref{e.bar_a_above} and~\eqref{e.bar_c} to bound the $\bar a^g$ and $\bar c^g$ terms in~\eqref{e.w03012} and using~\eqref{e.c52409} to remove several other terms, 
	we obtain, at $z_\eps$,
	\begin{equation}\label{e.w03015}
	\begin{split}
	2rW 
	&\lesssim
	W+\vv^{2\ell}|w||\bar a^w||D_{v}^{2}f|
	+
	\vv^{2\ell}|w||\bar c^w|f
	.
	\end{split}
	\end{equation}
	Recalling~\eqref{e.c033012} and~\eqref{e.bar_a_above}, we have, at $z_\epsilon$,
	\begin{equation}\label{e.w03013}
	\begin{split}
	|\bar a^w|
	\lesssim
	\vv^{(\gamma+2)_+}
	\|w\|_{L^{\infty,\ell}}
	=
	\vv^{(\gamma+2)_+}
	\sqrt W
	\end{split}
	\end{equation}
	and, by~\eqref{e.bar_c},
	\begin{equation}\label{e.w03014}
	\begin{split}
	|\bar c^w|
	\lesssim
	\vv^{\gamma}
	\|w\|_{L^{\infty,\ell}}
	=
	\vv^\gamma \sqrt W
	.
	\end{split}
	\end{equation}
	Plugging the estimates~\eqref{e.w03013} and~\eqref{e.w03014} into~\eqref{e.w03015} yields
	\begin{equation}\label{e.w03016}
	\begin{split}
	rW 
	&\lesssim
	W+ \vv^{\ell + (2+\gamma)_+} W |D_{v}^{2}f|
	+
	\vv^\ell W f
	\lesssim \left(1 + \vv^{\ell + (2+\gamma)_+}|D_v^2 f|\right) W.
	\end{split}
	\end{equation}
	Above we used that $\|f\|_{L^{\infty,\ell}}$ is bounded.  Applying \Cref{l.Hessian} and using~\eqref{e.c52511}, this becomes
	\be\notag
	rW
	\leq C_0
	\left(1+\frac{1}{t_\epsilon \log(\frac{1}{t_\epsilon})^{\frac{\theta'}{2}\frac{\alpha'}{2+\alpha'}}}
	\right)
	W
	\ee
	for some $C_0 >0$.  We note that this is where the restriction to $[0,T_0]$ is inherited from \Cref{l.Hessian}.

	Choosing 
	\begin{equation}\notag
	r(t)
	= 2C_0
	\left(1+\frac{1}{t \log(\frac{1}{t})^{\frac{\theta'}{2}\frac{\alpha'}{2+\alpha'}}}
	\right)
	\end{equation}
	contradicts \eqref{e.w03016}.  The condition that
	\be\notag
		\frac{\theta'}{2}\frac{\alpha'}{2+\alpha'} > 1
	\ee
	ensures that $r \in L^1([0,T_0])$, as desired.
	
	Therefore, this rules out the existence of  $z_\epsilon$.  We conclude that
	\be\notag
	\sup W < \epsilon.
	\ee
	As $\epsilon$ is arbitrary in the above argument, we deduce that $W\equiv 0$. Thus, $g=f$.
	
	It remains to address the case when $f_{\rm in} \in L^{\infty,k}$ for all $k$.  Here, however, the arguments of \cite[Theorem~1.4]{HST2019rough} directly apply.  Indeed, these arguments are based on showing that $f(T_1) \in C^{\alpha/3}_x C^\alpha_v(\R^6)$ and lies in $L^{\infty,k}$ for all $k$, which do not require the stronger smoothness assumptions of \cite[Theorem~1.2]{HST2019rough}.  The idea is to then re-apply the uniqueness argument on an interval starting at $T_1$.  Hence, we deduce that uniqueness on the entire time interval $[0,T]$.
\end{proof}

\subsection{A $t$-integrable bound on $\|D^2_v f(t)\|_{L^\infty_{x,v}}$: proof of \Cref{l.Hessian}}\label{s.Hessian}

We now state a more precise estimate that immediately yields \Cref{l.Hessian}.  It establishes a bound on $D^2_v f$ at the same time as one on the $\logalpha$-norm of $f$.  
In the sequel we refer to these as a Hessian bound and as propagation of regularity, respectively.  It is interesting to note that, although the latter is a `hyperbolic' estimate (that is, it does not involve a {\em gain of regularity}), it is dependent on the Schauder estimates in an essential way.

\begin{Proposition}
	\label{p.w02091}
	Let $f\in L^{\infty, k }([0,T]\times \R^6)$ be the solution constructed in \cite[Theorem~1.2]{HST2019rough} with the nondegeneracy condition~\eqref{e.nondegeneracy}.  Fix any $\theta > 0$, $\alpha \in (0,1)$, and $\mu < 1$ such that
	\be
		\frac{\mu\theta}{2}  \frac{\mu\alpha}{2+\mu \alpha}
			> 1
	\ee
	Then, for any $m> 5+\gamma$ and any $k$ sufficiently large depending on $m$, $\alpha$, $\theta$, and $\mu$, there exists a time $T_0 \leq \min\{1/2, T\}$ such that
	\begin{equation}
	\begin{split}
		\sup_{t\in[0,T_0]}&\left(t\left(\log\tfrac{1}{t}\right)^{\frac{\mu\theta}{2}\frac{\mu\alpha}{2+\mu\alpha}}\|D^2_v f\|_{L^{\infty,m+(2+\gamma)_+}([t/2,t]\times\R^6)}\right)^\frac{\mu\alpha}{2+\mu\alpha},
		\\&
		\|\vv^{m}f\|_{\logalphamu ([0,T_0]\times \R^6)}
		\lesssim
		1+
		\|f_{\text{in}}\|_{\logalpha ( \R^6)}.
	\end{split}
	\end{equation}
	As in \Cref{l.w02102}, the implied constant depends on $\|f\|_{L^{\infty,k}([0,T_0]\times\R^6)}$.  The final time $T_0$ depends only on $\alpha$, $\theta$, $k$, $\|f_0\|_{L^{\infty,k}}$, and $\|f_0\|_{\logalpha}$.
\end{Proposition}


We observe that the restriction $T_0 \leq 1/2$ is a technical one.  Indeed, one can iterate \Cref{p.w02091} starting at time $t=0$, $T_0$, $T_1$, $\dots$ to obtain the bound at some (potentially) large time.  As we see in its proof, and as is already hinted at by the exponent of the first term in the left hand side of the main inequality in \Cref{p.w02091}, it may be that the weighted H\"older norm blows up at a finite time.  We do not address this further here, as it was already handled at the conclusion of the proof of \Cref{t.Landau}.

In \cite{HST2019rough}, the analogue to \Cref{p.w02091} was broken up into two separate steps~\cite[Proposition~4.4 and Lemma~4.6]{HST2019rough}, one for each of the two inequalities.  Here, however, we must deduce the Hessian bound and the propagation of regularity simultaneously.    We discuss this in further detail after stating the next lemma, which plays a key role in the proof of \Cref{p.w02091}

The next lemma is an estimate on $D^2_v f$ in terms of the $\logalpha$-norm of $f$.  This is obtained by rescaling the equation, applying the Schauder estimates (\Cref{t.Schauder}), and then interpolating between the resulting $C^{\frac{2+\alpha}{3}}_xC^{2+\alpha}_v$-estimate and the $\logalpha$-seminorm of $f$.  

We note that, in order to do this, it is crucial that our Schauder estimates do not require $t$-H\"older regularity of the coefficients.  Indeed, the coefficient $\bar a^f$ is a $v$-convolution of $f$ and a kernel, and, hence, will have no more $t$-regularity than that of $f$.  {\em A priori} we do not have any bounds on the $t$-H\"older regularity of $f$.  One might attempt to obtain apply known estimates (e.g. the De Giorgi estimates \cite[Theorem 12]{GIMV}) to obtain a $t$-H\"older bound; however, these estimates will scale poorly in $t$, leading to a non-integrable bound in $t$.  This is overcome in~\cite[Proposition~A.1]{HST2019rough} by a lemma showing that $f$ obtains $t$ H\"older continuity from $(x,v)$ H\"older regularity.  This is clearly not useful in our setting as we do not yet have ``nice'' $v$ H\"older regularity of $f$.

\begin{Lemma}
	\label{l.w02102}
	Under the assumptions of \Cref{p.w02091}, 
\begin{equation}\begin{split}
	\|D_{v}^{2}f\|_{L^{\infty,m-2} ([t_0/2,t_0]\times \R^6)}
	\lesssim
	\ &\frac{1}{t_0 \log(\frac{1}{t_0})^{\frac{\theta}{2}\frac{\alpha}{2+\alpha}}}
	(1+	\|\vv^{m}f\|_{\logalpha([t_0/4, t_0]\times \R^6)})^{1+\frac{2}{\alpha}}
	\\&
		+
	\frac{t_0^{\alpha/2}}{\left(\log\frac{1}{t_0}\right)^{\frac{\alpha}{2(2+\alpha)}\frac{\theta}{2}}} \|D_v^2f\|_{L^{\infty,m-2}([t_0/4,t_0]\times \R^6)}^{\alpha/2}
	,
	\end{split}
	\end{equation}
	for any $t_0 \in (0,\min\{1/2,T\}]$.  The implied constant in the above estimate depends additionally on $\|f\|_{L^{\infty, k }([0,T]\times \R^6)}$.
\end{Lemma}

Again, we note that the fact that $t_0$ is restricted to be less than $1/2$ is only so that the $\log$ in the denominator does not take the value zero.

We now briefly comment that the necessity of proving both the Hessian bound and the propagation of regularity simultaneously is related to the fact that in \Cref{l.w02102}, one obtains both the $\logalpha$-norm  and a $W^{2,\infty}_v$-norm on the right hand side.  Hence, a dynamic argument is required in order to ``absorb'' the $W^{2,\infty}_v$-norm.  The reason that both terms appear in our setting (in contrast to the work in~\cite{HST2019rough}) is that we cannot bound the $\calpha$-norm of $\bar c^f$, which is required for the Schauder estimates, by the $\logalpha$-norm of $f$.

The proof of \Cref{l.w02102} is contained in \Cref{s.schauder_scaling}.


\subsection{The Hessian bound and propagation of regularity: the proof of \Cref{p.w02091}}

In this section, we prove the main estimate. Before that, we need to recast our notion of regularity.  For any point $(t,x,v,\chi, \nu)\in \R_{+}\times \R^6 \times B_{1/2}(0)^2$ and any real number $m>0$, we define
	\begin{equation}\label{e.w02091}
	\begin{split}
		&\tau f(t,x,v,\chi, \nu)
			:= f(t,x+\chi, v+\nu),
		\quad
		\delta f(t,x,v,\chi, \nu)
			:= f(t,x+\chi, v+\nu)-f(t,x,v),
		\\&
		\text{and }\quad
		g(t,x,v,\chi, \nu)
			:= \frac{|\delta f(t,x,v,\chi, \nu)|^2}{(|\chi|^2 + |\log|\nu||^{-2\theta/\alpha})^{\mu \alpha}}\vv^{2m}.
	\end{split}
	\end{equation}
	Then we have the following obvious equivalence between bounds on $g$ and the regularity of $f$.  We omit the proof.
\begin{Lemma}
	\label{l.w02101}
	We have
	\begin{equation}\notag
	\begin{split}
	\|g\|_{L_{x,v}^{\infty}} + \|\vv^{m}f\|_{L^{\infty}(\R^6)}^2
	&\approx
	\|\vv^{m}f\|_{\logalphamu(\R^6)}^2
	\\&\approx
	\sup_{(x_0,v_0)}
	\vvo^{m}\|f\|_{\logalphamu(B_{1/2}(x_0,v_0))}^2
	,
	\end{split}
	\end{equation}
	where the implied constant depend only on $m$, $\theta$, and $\alpha$.
\end{Lemma}

With \Cref{l.w02101} in hand, we are now able to prove our main estimate \Cref{p.w02091} 
using the strategy of~\cite[Proposition~4.4]{HST2019rough}.  When the details are the same as in \cite[Proposition~4.4]{HST2019rough} we note this and omit them.

\begin{proof}[Proof of \Cref{p.w02091}]
	Before beginning we note two things.  The first is that, since we are proving a statement regarding a solution constructed in \cite{HST2019rough}, we may assume  without loss of generality that $f$ is smooth.  Indeed, in~\cite{HST2019rough}, the solution $f$ is approximated by smooth solutions of~\eqref{e.Landau}.  Were we to prove the claim for the approximating solution, it holds for $f$ in the limit.

	Next, we note that $f \in L^{\infty,k}$, by assumption.  Hence, we ignore this norm throughout and absorb all instances of it into the $\lesssim$ notation.

	As the proof is somewhat complicated, we break it up into a number of steps.

	\smallskip
	
	\noindent	
	{\bf Step 1: an equation for $g$ and straightforward estimates.} Using~\eqref{e.Landau}, we find
	\begin{equation}\label{e.c51501}
	\begin{split}
	&\partial_{t}g + v\cdot \nabla_{x}g +\nu\cdot \nabla_{\chi}g 
	+
	\frac{2\alpha \mu \chi \cdot \nu}{|\chi|^2 + |\log|\nu||^{-2\theta/\alpha}} g
	\\&\qquad
	=
	2\frac{\tr(\bar a^{\delta f}D_{v}^{2}\tau f+\bar a^{ f}D_{v}^{2}\delta f)
		+\bar c^{\delta f}\tau f
		+\bar c^{f}\delta f}{(|\chi|^2 + |\log|\nu||^{-2\theta/\alpha})^{\mu\alpha}}\delta f \vv^{2m}.
	\end{split}
	\end{equation}
	Three terms are estimated exactly\footnote{This corresponds to the estimates of $J_1$, $J_4$, and $J_5$ in \cite[Proposition~4.4]{HST2019rough}.} as in \cite[Proposition 4.4]{HST2019rough}:
	\be\label{e.c51502}
		- \frac{2\alpha \mu \chi \cdot \nu}{|\chi|^2 + |\log|\nu||^{-2\theta/\alpha}} g
		+\frac{\bar c^{\delta f}\tau f
		+\bar c^{f}\delta f}{(|\chi|^2 + |\log|\nu||^{-2\theta/\alpha})^{\mu\alpha}}\delta f \vv^{2m}
			\lesssim g + \sqrt{ g \|g(t)\|_{L^\infty(\R^6\times B_1^2)}}.
	\ee
	Here we used~\eqref{e.bar_c}, the condition that $m > 5+\gamma$, and that the $L^{\infty,k}$-norm of $f$ bounds $\vv^m \tau f$.

	Additionally, arguing as in \cite[Proposition 4.4]{HST2019rough}, one sees
	\be\notag
		\frac{|\bar a^{\delta f}|}{(|\chi|^2 + |\log|\nu||^{-2\theta/\alpha})^{\mu\alpha/2}}
			\lesssim \vv^{(2+\gamma)_+} \sqrt{\|g(t)\|_{L^\infty(\R^6\times B_1^2)}}.
	\ee
	The argument for this uses the definition of $g$ in terms of $\delta f$ and~\eqref{e.bar_a_above}.

	Hence, we have
	\be\label{e.c52405}
		\begin{split}
		\partial_t g + v\cdot\nabla_x g + &\nu\cdot \nabla_\chi g
			- 2\frac{\tr(\bar a^{ f}D_{v}^{2}\delta f)}{(|\chi|^2 + |\log|\nu||^{-2\theta/\alpha})^{\mu\alpha}}\delta f \vv^{2m}
			\\&
			\lesssim g + \left(1 + \|D_{v}^{2}\tau f(t)\|_{L^{\infty,m+(2+\gamma)_+}(\R^6)}\right) \sqrt{g\|g(t)\|_{L^\infty(\R^6 \times B_1^2)}}.
		\end{split}
	\ee
	This concludes the first step.
	
	To briefly comment on how we proceed from here, note that, roughly the terms on the left hand side should have a good sign at a maximum of $g$ (if we think of $\delta f$ as, approximately $\sqrt g$).  On the other hand, the pure $g$ term on the right hand side lend itself to the construction of a barrier.  The most complicated term is the Hessian term in $\tau f$.  For this, we use \Cref{l.w02102} and the fact that the Hessian term on the left has a small parameter in front (when $t_0 \ll 1$), which, through a somewhat complicated process, allows us to to absorb this into Hessian in the left hand side of \Cref{l.w02102}.

	\smallskip
	
	\noindent	
	{\bf Step 2: an upper barrier.} 
	With $N > 1$ to be chosen later and fixing any $0< \mu' < \mu$  such that
	\be
		\frac{\mu' \theta}{2} \frac{\mu' \alpha}{2 + \mu'\alpha}
			> 1,
	\ee
	define $\bar G$ to be the unique solution to 
	\begin{equation}\label{e.w02281}
	\begin{cases}
		\frac{d}{dt} \bar G(t)
			= \frac{N^2}{t \left(\log\frac{1}{t}\right)^{\frac{\mu\theta}{2}\frac{\mu\alpha}{2+\mu\alpha}}} (1 + \bar G)^{\frac{1}{2} + \frac{1}{\mu' \alpha}} \bar G,
	\\
		\bar G(0)=\|g(0)\|_{L^\infty} + N\|f\|_{L^{\infty, m}}^2+1
	.
	\end{cases}
	\end{equation}
	Let $T_1$ be the largest time in $[0,1/2]$ that $\bar G(T_1) \leq 2 \bar G(0)$.  Let
	\be
		T_0 = \min\{1, T_1, T_2\}
	\ee
	for $T_2$ to be chosen in the sequel.  Clearly $T_1$ depends on $N$, but $N$ will be chosen to depend only on $\mu$, $\alpha$, $\theta$, $m$, and $k$.  We note that
	\be\label{e.c52601}
		\bar G(t) \geq 1
			\qquad\text{ for all } t\in [0,T_0].
	\ee
	
	We define the auxiliary function
	\be
		G_2(t)
			= 
			t \left(\log \frac{1}{t}\right)^{\frac{\mu'\theta}{2}\frac{\mu'\alpha}{2 + \mu'\alpha}} \|D^2_v f\|_{L^{\infty,m+(2+\gamma)_+}([t/2,t]\times\R^6)},
	\ee
	and then let
	\be
		G(t, x, v, \chi, \nu)
			= \max\left\{g(t,x,v,\chi,\nu),  \Big(\frac{1}{N} G_2(t)\Big)^\frac{2\mu' \alpha}{2 + \mu'\alpha}\right\}.
	\ee
	Our goal is to show that, for $t\in [0,T_0]$,
	\begin{equation}
		G(t,x,v,\chi,\nu) < \bar G(t).
	\end{equation}
	This is true at $t=0$ by construction (recall that, without loss of generality, our $f$ is smooth, so that $G_2(0) = 0$).  Hence, we may define
	\be\label{e.c52301}
		t_0 = \sup\{\bar t \in[0,T_0]: \|G(s)\|_{L^\infty(\R^6)} <  \bar G(s) \quad\text{for all } s \in [0,\bar t]\}.
	\ee
	If $t_0 = T_0$, we are finished.  Hence, we argue by contradiction, assuming that
	\be\label{e.c51505}
		t_0 < T_0.
	\ee
	
	\smallskip
	
	\noindent{\bf Step 3: The case where $g$ is not the dominant term in $G$.} 
	Clearly $\|G(t_0)\|_{L^\infty(\R^6 \times B_1^2)} = \bar G(t_0)$.  Consider the case where
	\be\label{e.c52402}
		\|g(t_0)\|_{L^\infty(\R^6 \times B_1^2)}
			< \Big(\frac{1}{N} G_2(t_0)\Big)^\frac{2\mu' \alpha}{2 + \mu'\alpha}
		\quad\text{ so that }
		\qquad \Big(\frac{1}{N} G_2(t_0)\Big)^\frac{2\mu' \alpha}{2 + \mu'\alpha} = \bar G(t_0).
	\ee
	
	
	Then, using \Cref{l.w02102} and that
	\be
		t_0 \left(\log \frac{1}{t_0}\right)^{\frac{\mu'\theta}{2}\frac{\mu'\alpha}{2 + \mu'\alpha}}
			\|D^2_v f\|_{L^\infty([t_0/4,t_0]\times \R^6)}
				\lesssim 
				 G_2(t_0/2) + G_2(t_0),
	\ee
	we find
	\be
	\begin{split}
		G_2(t_0)
			\lesssim\ 
				&\left(1 + \|\vv^{m+2+(2+\gamma)_+} f\|_{\logalphamup([t_0/4,t_0]\times \R^6)}\right)^{1 + \frac{2}{\mu'\alpha}}
				\\& \quad
				+ t_0^\frac{\mu' \alpha}{2}
					\left(G_2(t_0/2) + G_2(t_0)\right)^\frac{\mu'\alpha}{2}.
	\end{split}
	\ee
	Using the interpolation lemma (\Cref{l.weight_interpolation}), \Cref{l.w02101}, and increasing $k$ if necessary, we find
	\be\label{e.c52401}
	\begin{split}
		G_2(t_0)
			\lesssim\ 
				&\left(1 + \|g\|^{1/2}_{L^\infty([t_0/4,t_0]\times \R^6 \times B_1^2)}\right)^{1 + \frac{2}{\mu'\alpha}}
				+ t_0^\frac{\mu' \alpha}{2}\left(G_2(t_0/2)^\frac{\mu'\alpha}{2} + G_2(t_0)^\frac{\mu'\alpha}{2}\right).
	\end{split}
	\ee
	We recall that we are not tracking the $L^{\infty,k}$-norm of $f$ as it is bounded by assumption.  We also note that it is in this step that we used that $\mu' < \mu$.

	By the choice of $t_0$ and the fact that $\bar G$ is increasing, we see that
	\be
		G_2(t_0/2) \leq N\bar G(t_0/2)^\frac{2+\mu'\alpha}{2\mu'\alpha}
			\leq N \bar G(t_0)^\frac{2+\mu'\alpha}{2\mu'\alpha}
			= G_2(t_0).
	\ee
	Also, by the definition of $t_0$ and~\eqref{e.c52402},
	\be
		\|g\|_{L^\infty([t_0/4,t_0]\times \R^6 \times B_1^2)}
			\leq \sup_{t\in[t_0/4,t_0]} \bar G(t)
			= \bar G(t_0)
			= \Big(\frac{1}{N} G_2(t_0)\Big)^\frac{2\mu'\alpha}{2 + \mu'\alpha}.
	\ee
	Then~\eqref{e.c52401} becomes:
	\be
		\begin{split}
		G_2(t_0)
			&\lesssim \left(1 + \Big(\frac{1}{N} G_2(t_0)\Big)^\frac{\mu' \alpha}{2+\mu'\alpha}\right)^\frac{2 + \mu'\alpha}{\mu'\alpha}
				+ t_0^\frac{\mu' \alpha}{2}
					 G_2(t_0)^\frac{\mu'\alpha}{2}
			\lesssim
				1
				+ \frac{1}{N} G_2(t_0)
				+ t_0^\frac{\mu' \alpha}{2} 
					G_2(t_0)^\frac{\mu'\alpha}{2}.
		\end{split}
	\ee
	Since $G_2(t_0) = N \bar G(t_0)^\frac{2+\mu'\alpha}{2\mu'\alpha} >1$, we have that	
	\be
		G_2(t_0)
			\lesssim 1+ \frac{1}{N} G_2(t_0)
				+ t_0^\frac{\mu' \alpha}{2}G_2(t_0).
	\ee
	After increasing $N$ and decreasing $T_2$, we may absorb the last two terms on the right into the left hand side.  After this and recalling~\eqref{e.c52402}, we find
	\be
		(N \bar G(t_0))^\frac{2+\mu'\alpha}{\mu'\alpha}
			= G_2(t_0)
			\lesssim 1.
	\ee
	After further increasing $N$ and recalling~\eqref{e.c52601}, this is clearly a contradiction.  It follows that~\eqref{e.c52402} cannot hold.  We conclude that
	\be\label{e.c52403}
		\|g(t_0)\|_{L^\infty(\R^6\times B_1^2)}
			= \bar G(t_0).
	\ee
	An important consequence of this is that, for all $t \leq t_0$,
	\be\label{e.c52404}
		\|D^2_v f\|_{L^{\infty,m+(2+\gamma)_+}([t/2,t]\times\R^6)}
			\leq \frac{N}{t \left(\log \frac{1}{t}\right)^{\frac{\mu'\theta}{2}\frac{\mu'\alpha}{2+\mu'\alpha}}}
				\|g(t)\|_{L^\infty(\R^6 \times B_1^2)}^{\frac{1}{2} + \frac{1}{\mu' \alpha}}.
	\ee

	\smallskip
	
	\noindent{\bf Step 4: The bad Hessian term in~\eqref{e.c52405} and an interpolation.}
	We now use~\eqref{e.c52404} in~\eqref{e.c52405} to bound the norm of the Hessian that arises there.
	
	We require one additional fact.  By the choice of $t_0$ and by~\eqref{e.c52403}, we have
	\be\label{e.c52406}
		\|g\|_{L^\infty([0,t_0]\times \R^6\times B_1^2)}
			= \|g(t_0)\|_{L^\infty(\times \R^6\times B_1^2)}.
	\ee
	Thus, at $(t_0,x_0,v_0,\chi_0,\nu_0)$, the combination of~\eqref{e.c52404} and~\eqref{e.c52406} in~\eqref{e.c52405} yields
	\be\label{e.c52407}
		\begin{split}
		\partial_t g + &v\cdot\nabla_x g + \nu\cdot \nabla_\chi g
			- 2\frac{\tr(\bar a^{ f}D_{v}^{2}\delta f)}{(|\chi|^2 + |\log|\nu||^{-2\theta/\alpha})^{\mu\alpha}}\delta f \vv^{2m}
			\\&
			\lesssim g + \frac{N}{t_0 \left(\log \frac{1}{t_0}\right)^{\frac{\mu'\theta}{2}\frac{\mu'\alpha}{2+\mu'\alpha}}} g^{\frac{3}{2} + \frac{1}{\mu'\alpha}}
			\lesssim \frac{N}{t_0 \left(\log \frac{1}{t_0}\right)^{\frac{\mu'\theta}{2}\frac{\mu'\alpha}{2+\mu'\alpha}}}(1 + g^{\frac{1}{2} + \frac{1}{\mu'\alpha}})g.
		\end{split}
	\ee

	\smallskip
	
	\noindent{\bf Step 5: finding a touching point.}   Using~\eqref{e.c52403} and arguing exactly as in the proof of \cite[Proposition~4.4]{HST2019rough}, we may assume without loss of generality that there exists $(x_0,v_0,\chi_0,\nu_0) \in \R^6 \times \bar B_1 (0)^2$ such that
	\be\label{e.c51203}
		g(t_0, x_0,v_0,\chi_0,\nu_0)
			=\bar G(t_0).
	\ee
	We omit the argument.
	
	\smallskip
	
	\noindent{\bf Step 6: the touching point must be in $B_1(0)^2$.}  If $\chi_0$ or $\nu_0$ were on the boundary, that is, either $\chi_0 \in \partial B_1 (0)$ or $\nu_0 \in \partial B_1(0)$, we deduce from the definition of $g$ that
	\begin{equation}\label{e.c51202}
	\begin{split}
	g(t_0, x_0, v_0,\chi_0,\nu_0)&\lesssim
	|\delta f(t_0, x_0, v_0,\chi_0,\nu_0)|^2
	\langle v_0 \rangle^{2m}
	\\&\lesssim
	(f(t_0, x_0+\chi_0,v_0 +\nu_0)^2
	+
	f(t_0, x_0, v_0)^2)\langle v_0 \rangle^{2m}
	\lesssim
	\|f\|^{2}_{L^{\infty, m}}.
	\end{split}
	\end{equation}
	In particular, this implies that, up to enlarging $N$ large enough depending only on the implied constant in~\eqref{e.c51202},
	\begin{equation}\notag
	\begin{split}
	g(t_0, x_0, v_0,\chi_0,\nu_0)\leq
	N\|f\|^{2}_{L^{\infty, m}}
	.
	\end{split}
	\end{equation}
	We see from \eqref{e.w02281} that $\bar G$ increases with time $t$. Thus,
	\begin{equation}\notag
	\bar G(t_0)
	\geq 
	\bar G(0)
	>N\|f\|^{2}_{L^{\infty, m}}
	,
	\end{equation}
	which contradicts~\eqref{e.c51203}.
	
	\smallskip
	
	\noindent
	{\bf Step 7: estimating the remaining term in~\eqref{e.c52407}.}
	We begin by expanding the last term on the left hand side of~\eqref{e.c52407} at the point $(t_0,x_0,v_0,\chi_0,\nu_0)$.  This is a simple multivariable calculus computation that is exactly as in \cite[Proposition~4.4]{HST2019rough}, so we omit it and simply state that:
	\be\notag
		\frac{\tr(\bar a^f D^2_v \delta f)}{(|\chi|^2 + |\log|\nu||^{-2\theta/\alpha})^{\mu\alpha}} \delta f \vv^{2m}
			= \tr(\bar a^f D^2_v g)
				+ \frac{2 m g}{\vvo^4} \left((m+2)v_0 \cdot \bar a^f v_0 - \vvo^2\tr \bar a^f\right).
	\ee
	This argument occurs at and below (4.7) in \cite{HST2019rough}.
	
	Since $g$ is at a maximum, we further obtain
	\be\notag
		\frac{\tr(\bar a^f D^2_v \delta f)}{(|\chi|^2 + |\log|\nu||^{-2\theta/\alpha})^{\mu\alpha}} \delta f \vv^{2m}
			\leq \frac{2 m g}{\vvo^4} \left((m+2)v_0 \cdot \bar a^f v_0 - \vvo^2\tr \bar a^f\right).
	\ee
	Hence, arguing as in \cite[Proposition~4.4]{HST2019rough} to bound the terms on the right hand side above, we find\footnote{This is the estimate of $J_3$ in \cite{HST2019rough}.  It is somewhat obvious from~\eqref{e.bar_a_above}.}
	\be\notag
		\frac{\tr(\bar a^f D^2_v \delta f)}{(|\chi|^2 + |\log|\nu||^{-2\theta/\alpha})^{\mu\alpha}} \delta f \vv^{2m}
			\lesssim g.
	\ee
	
	Combining the above with~\eqref{e.c52407},
	we have, at $(t_0, x_0,v_0, \chi_0,\nu_0)$,
	\be\label{e.c51504}
		\partial_t g + v\cdot\nabla_x + \nu \nabla_\chi g
			\lesssim
			\frac{N}{t \left(\log \frac{1}{t}\right)^{\frac{\mu'\theta}{2}\frac{\mu'\alpha}{2+\mu'\alpha}}}
				\left(1 + g\right)^{\frac{1}{2} + \frac{1}{\mu' \alpha}} 
				g.
	\ee

	\smallskip
	
	\noindent
	{\bf Step 8: concluding the proof.}  By the construction of $(t_0,x_0,v_0,\chi_0,\nu_0)$, it is a minimum of $\bar G - g$ on $[0,t_0]\times \R^6 \times B_1^2$.  Hence,
	\be\notag
		\partial_t (\bar G - g)
			+ v\cdot\nabla_x (\bar G - g)
			+ \nu \cdot \nabla_\chi(\bar G - g)
			\leq 0.
	\ee
	Using~\eqref{e.w02281} and~\eqref{e.c51504} and recalling that $\bar G(t_0) = g(t_0,x_0,v_0,\chi_0,\nu_0)$, this implies that, at $(t_0,x_0,v_0,\chi_0,\nu_0)$,
	\be\notag
		\begin{split}
		\frac{N^2}{t_0 \left(\log \frac{1}{t_0}\right)^{\frac{\mu' \theta}{2} \frac{\mu' \alpha}{2 + \mu'\alpha}}} \left(1 + \bar G\right)^{\frac{1}{2} + \frac{1}{\mu'\alpha}} \bar G
			&\lesssim \frac{N}{t \left(\log \frac{1}{t}\right)^{\frac{\mu'\theta}{2}\frac{\mu'\alpha}{2+\mu'\alpha}}}
				\left(1 + g\right)^{\frac{1}{2} + \frac{1}{\mu' \alpha}} 
				g
			\\&
			= \frac{N}{t \left(\log \frac{1}{t}\right)^{\frac{\mu'\theta}{2}\frac{\mu'\alpha}{2+\mu'\alpha}}}
				\left(1 + \bar G\right)^{\frac{1}{2} + \frac{1}{\mu' \alpha}} 
				\bar G.
		\end{split}
	\ee
	This is a contradiction if $N$ is sufficiently large.  Hence, it must be that~\eqref{e.c51505} does not hold, implying that
	\be\notag
		\sup_{(x,v,\chi,\nu)\in \R^6\times B_1^2} g(t,x,v,\chi,\nu) \leq \bar G(t)
			\qquad\text{ for all } t\in [0,T_0]
	\ee
	by definition of $t_0$.  Recalling \Cref{l.w02101}, this concludes the proof of the bound of $\|\vv^m f\|_{\calphamu([0,T_0]\times \R^6)}$.  The proof of the bound on the Hessian term in \Cref{p.w02091} follows from~\eqref{e.c52404} and the arbitrariness of $\mu$ and $\mu'$.
\end{proof}

\subsection{Scaling the Schauder estimates: proof of \Cref{l.w02102}}\label{s.schauder_scaling}

Due to the degeneracy of the ellipticity constants of $\bar a^f$ as $|v|\rightarrow \infty$ and the fact that $Q_1(t_0,x_0,v_0)$ may involve negative times, we must change of variables.  We begin by defining this change of variables.  It is the one used in~\cite{Cameron-Silvestre-Snelson,HS,HST2019rough}.

Fix $z_0 \in \R_{+}\times \R^{6}$. Let $S$ be the linear transformation such that
\begin{equation}\label{e.w05201}
Se = 
\begin{cases}
\langle v_{0}\rangle^{1+\gamma/2}e,& e\cdot v_{0}=0\\
\langle v_{0}\rangle^{\gamma/2}e,& e\cdot v_{0}=|v_0|,
\end{cases}
\end{equation}
and let
\begin{equation}\label{e.w05202}
r_0 =\langle v_{0}\rangle^{-(1+\gamma/2)_{+}}\min(1,\sqrt{t_0/2})
.
\end{equation}
Then we have the rescaled function
\be\label{e.f_z_0}
	f_{z_0}(z)
		:= f(r_0^2 t + t_0, r_0^3 Sx + x_0 + r_0^2 t v_0, r_0 S v + v_0),
\ee
which satisfies the rescaled equation
\be\notag
	(\partial_t + v\cdot\nabla_x) f_{z_0}
		= \tr(\bar A D^2_v f_{z_0})
			+ \bar C f_{z_0}
\ee
with coefficients
\be\notag
	\begin{split}
		&\bar A(z)
			= S^{-1} \bar a^f(r_0^2 t + t_0, r_0^3 S x + x_0+ r_0^2 t v_0, r_0 S v + v_0) S^{-1}
		\quad\text{ and}
		\\&
		\bar C(z)
			= r_0^2 \bar c^f(r_0^2 t + t_0, r_0^3 S x + x_0+ r_0^2 t v_0, r_0 S v + v_0).
	\end{split}
\ee
Roughly, the input of $f$ in the definition of $f_{z_0}$ can be written as $z_0 \circ (Sz)_{r_0}$ where $z_r = (r^2 t, r^3 x, r v)$ is the kinetic scaling by a factor $r$ and
\be
	z' \circ z
		= (t' + t, x' + x + t v', v' + v)
\ee
is the related to the Galilean Lie group structure associated to $\partial_t + v\cdot\nabla_x$.  For simplicity, we opt not to use this further, although it is common in the literature.

It is immediate from~\eqref{e.bar_a_above},~\eqref{e.bar_a_below}, and \cite[Proposition~3.1]{HS} 
 that
\be\label{e.c51204bis}
	\bar A
		\approx \Id
		\qquad\text{ on } Q_1,
\ee
and, by an easy computation (see~\cite[eqn~(2.15)]{HST2019rough}),
\be\label{e.c51204bisbis}
	\bar C(z)
		\lesssim \vvo^{-2} \min\{1, t_0\} \|f\|_{L^{\infty,m}}
\ee
for any $m > 3$.  Additionally, one can observe that
\be\label{e.c51413}
	\|f_{z_0}\|_{L^\infty(Q_1)}
		\lesssim \vvo^{-k} \|f\|_{L^{\infty,k}}.
\ee
We omit the proof of the above inequalities as they are straightforward and already contained in \cite{HS, HST2019rough}.

We note that the coefficients have the following regularity:
\begin{Lemma}
	\label{l.w01161}
	For $m, k> 5 + \gamma$ and $\alpha \in (0,1)$, we have
	\begin{equation}\label{e.w01241}
		[\bar A]_{\calpha(Q_{3/4})}
		\lesssim
		t_0^\frac{\alpha}{2} \left(\vvo^{2-\alpha} \|f\|_{L^{\infty,k}([t_0/4,t_0]\times \R^6)}
				+ \vvo^2 \|\vv^m f\|_{C_x^{\alpha/3}([t_0/4,t_0]\times \R^6)}\right)
	\end{equation}
	and
	\be\notag
		[\bar C]_{\calpha(Q_{3/4})}
			\lesssim t_0^{1+\frac{\alpha}{2}} 
			\vvo^\gamma \|\vv^m f\|_{\calpha([t_0/4,t_0]\times \R^6)}.
	\ee
\end{Lemma}
We note that \Cref{l.w01161} is stronger than its analogue \cite[Lemma~2.7]{HST2019rough} as we leverage the convolutional nature of $\bar a^f$ to obtain additional regularity in $v$ even when $f$ lacks regularity in $v$.  Additionally, the fact that we do not require $t$-regularity allows us to avoid the slight loss of regularity seen in \cite[Lemma~2.7]{HST2019rough}.  On the other hand, we note that we make no effort to optimize the $v_0$-weights in \Cref{l.w01161}.  We prove \Cref{l.w01161} in \Cref{s.technical}.

Moreover, we immediately see that the regularity of $f_{z_0}$ and $f$ are related by:
\be\label{e.c51417}
	\|f\|_{\calpha(Q_{r_0/2}(z_0))}
		\lesssim \min\{1, t_0\}^{-\alpha/2}\vvo^{\alpha( (1 + \gamma/2)_+ - \gamma/2)} \|f_{z_0}\|_{\calpha(Q_{1/2})}.
\ee
Analogous statements hold for higher regularity seminorms of $f$ as well.  Here, we are introducing the additional notation that
\be\label{e.c52410}
	Q_r(z_0)
		= \{(t,x,v) : t_0 - r^2< t\leq t_0, |x-x_0 - (t-t_0)v_0| < r^3, |v-v_0| < r\}.
\ee

Finally, before proving \Cref{l.w02102}, we state two final technical results related to scaling:
\begin{Lemma}[$\log$-H\"older interpolation inequality]\label{l.w01081}
	Fix any $u : \R^d \to \R$ and $r > 0$.  For $\alpha \in (0,1)$, $\theta>0$, and any $\eps\in(0,r)$,
	\begin{equation}\notag
	\begin{split}
	\|D_{v}^{2}u\|_{L^{\infty}(Q_r)}
	\lesssim
	\frac{\log(\eps)^{\theta}}{\eps^2}[u]_{\log(\frac{1}{C})^{-\theta}(Q_r)}
	+
	\eps^\alpha [u]_{C^{2,\alpha}(Q_r)}
	.
	\end{split}
	\end{equation}
	The implied constant depends only on $\theta$ and $\alpha$.
\end{Lemma}
\begin{Lemma}\label{l.logH_scaling}
	We have, for $t_0 < 1/2$,
	\be\label{e.c022501}
		[f_{z_0}]_{\log(\frac{1}{C})^{-\theta/2}(Q_1)}
		\lesssim [f]_{\log(\frac{1}{C})^{-\theta}(Q_{t_0/2}(z_0))} \log\Big(\frac{1}{t_0}\Big)^{-\theta/2}.
	\ee
\end{Lemma}
The proofs of these two lemmas are also postponed to \Cref{s.technical}. 
We now prove the lemma on the scaling of the Schauder estimates.
\begin{proof}[Proof of \Cref{l.w02102}]
Throughout the proof we assume that
\be\notag
	[\vv^m f]_{\logalpha([t_0/4,t_0]\times \R^6)}
		< \infty.
\ee
If this were not true, then the claim in \Cref{l.w02102} follows immediately.

Fix $\eps\in(0,1/2)$ to be determined.  Applying our $\log$-H\"older interpolation lemma (\Cref{l.w01081}), we see
\begin{equation}\label{e.c51410}
	\|D_{v}^{2}f_{z_0}\|_{L^\infty (Q_{1/2})}
		\lesssim
			\frac{\log(1/\eps)^{-\theta/2}}{\eps^2}[f_{z_0}]_{\log(\frac{1}{C_v})^{-\theta/2}(Q_{1/2})}
			+ \eps^{\alpha}[D^2_v f_{z_0}]_{\calpha(Q_{1/2})}.
\end{equation}
Clearly the first term in~\eqref{e.c51410} can be bounded by simply removing the scaling.  Indeed, applying \Cref{l.logH_scaling}, we find
\be\label{e.c51411}
	[f_{z_0}]_{\log(\frac{1}{C})^{-\theta/2}}
		\lesssim  \left( \log \frac{1}{t_0}\right)^{-\theta/2}
			[f]_{\logalpha(Q_{t_0/2}(z_0))}.
\ee

For the second term in~\eqref{e.c51410}, we require our Schauder estimates \Cref{t.Schauder}.  Applying this yields
\be\notag
	[D^2_v f_{z_0}]_{\calpha(Q_{1/2})}
		\lesssim \left(
			1 + [\bar C]_{\calpha(Q_{3/4})}
			+ [\bar A]_{\calpha(Q_{3/4})}^{1 + \frac{2}{\alpha}}
			\right)
			\|f_{z_0}\|_{L^\infty(Q_{3/4})}.
\ee
We note that the statement of \Cref{t.Schauder} involves a cylinder $Q_1$ on the right hand side instead of $Q_{3/4}$; however, it is a simple scaling argument to obtain the above, so we omit the details.  We use this cylinder in order to obtain an estimate below insulated from $t=0$ by $t_0/4$.

Using~\eqref{e.c51413} and \Cref{l.w01161}, we obtain
\be\notag
	\begin{split}
	&[D^2_v f_{z_0}]_{\calpha(Q_1)}
		\lesssim \vvo^{-k} \bigg(
			1 + t_0^{1+\frac{\alpha}{2}} \vvo^{\gamma} \|\vv^m f\|_{C^\alpha_v([t_0/4,t_0])} + 
			\\&
			\phantom{MMMMMMMMM}
			\left(t_0^\frac{\alpha}{2} \vvo^{2-\alpha} \|f\|_{L^{\infty,k}([t_0/4,t_0])}
			+ t_0^\frac{\alpha}{2} \vvo^2 \|\vv^m f\|_{C^{\alpha/3}_x([t_0/4,t_0])}
			\right)^{1 + \frac2{\alpha}}\bigg)
			\|f\|_{L^{\infty,k}}.
	\end{split}
\ee
We recall, by assumption, $\|f\|_{L^{\infty,k}}$ is finite.  This is inherited from \cite[Theorem~1.2]{HST2019rough}.  Hence,
\be\label{e.c51416}
	\begin{split}
		[D^2_v f_{z_0}]_{\calpha(Q_1)}
		\lesssim\
			&\vvo^{-k + \gamma} t_0^{1 + \frac{\alpha}{2}} [\vv^m f]_{C^\alpha_v([t_0/4,t_0])}		
		\\&
		+ \vvo^{-k + 2 + \alpha} \left(
			1 
			+ \|\vv^m f\|_{C^{\alpha/3}_x([t_0/4,t_0])}
			\right)^{1 + \frac2{\alpha}}.
	\end{split}
\ee

Using~\eqref{e.c51411} and~\eqref{e.c51416} in~\eqref{e.c51410}, we find
\be\notag
	\begin{split}
		\|D^2_v f_{z_0}\|_{L^\infty(Q_{1/2})}
			\lesssim &\frac{\log(1/\eps)^{-\theta/2}}{\eps^2}
				\left(\log \frac{1}{t_0}\right)^{-\frac{\theta}{2}} [f]_{\logalpha(Q_{t_0/4}(z_0))}
			\\&
				+ \eps^{\alpha} \vvo^{-k + 2 + \alpha} \left(
			1 
			+ \|\vv^m f\|_{C^{\alpha/3}_x([t_0/4,t_0])}
			\right)^{1 + \frac2{\alpha}}
			\\& + \eps^\alpha \vvo^{-k + \gamma} t_0^{1 + \frac{\alpha}{2}} [\vv^m f]_{C^\alpha_v([t_0/4,t_0])}.
	\end{split}
\ee
Undoing the change of variables (similar to~\eqref{e.c51417}) and combining terms yields
\be\label{e.c51418}
	\begin{split}
		\frac{t_0}{\vvo^2}
			\|D^2_v f\|_{L^\infty(Q_{t_0/2}(z_0))}
			\lesssim\ &\frac{\log(1/\eps)^{-\theta/2}}{\eps^2}
				\left(\log \frac{1}{t_0}\right)^{-\frac{\theta}{2}} [f]_{\logalpha(Q_{t_0/4}(z_0))}
			\\&
				+ \eps^{\alpha} \vvo^{-k + 2 + \alpha} \left(
			1 
			+ \|\vv^m f\|_{C^{\alpha/3}_x([t_0/4,t_0]\times \R^6)}
			\right)^{1 + \frac2{\alpha}}
			\\& + \eps^\alpha \vvo^{-k + \gamma} t_0^{1 + \frac{\alpha}{2}} [\vv^m f]_{C^\alpha_v([t_0/4,t_0])}.
	\end{split}
\ee

Next, we take
\be\notag
	\eps
		= \min\left\{1/4, \log(1/t_0)^{-\frac{\theta}{2(2+ \alpha)}}\right\}
\ee
so that~\eqref{e.c51418} becomes
\be\notag
	\begin{split}
	\frac{t_0}{\vvo^2}
			&\|D^2_v f\|_{L^\infty(Q_{t_0/2}(z_0))}
			\lesssim \left( \log \frac{1}{t_0}\right)^{-\frac{\alpha}{2(2 + \alpha)}\frac{\theta}{2}} [f]_{\logalpha(Q_{t_0/4}(z_0))}
			\\&
				+ \left( \log \frac{1}{t_0}\right)^{-\frac{\alpha}{2(2 + \alpha)}\frac{\theta}{2}}\vvo^{-k + 2 + \alpha} \left(
			1 
			+ \|\vv^m f\|_{C^{\alpha/3}_x([t_0/4,t_0]\times \R^6)}
			\right)^{1 + \frac2{\alpha}}
			\\& + \left( \log \frac{1}{t_0}\right)^{-\frac{\alpha}{2(2 + \alpha)}\frac{\theta}{2}} \vvo^{-k + \gamma} t_0^{1 + \frac{\alpha}{2}} [\vv^m f]_{C^\alpha_v([t_0/4,t_0])}.
	\end{split}
\ee
Dividing by $t_0$, multiplying by $\vvo^m$, increasing $k$ if necessary, and taking the supremum over all choices of $(x_0,v_0)$, we find
\be\notag
	\begin{split}
	\|D^2_v f\|_{L^{\infty,m-2}([t_0/2,t_0])}
		\lesssim\
			 &\frac{1}{t_0 \left( \log \frac{1}{t_0}\right)^{\frac{\alpha}{2(2 +  \alpha)}\frac{\theta}{2}}} \left(
				1 
				+ \|\vv^m f\|_{\logalpha([t_0/4,t_0]\times \R^6)}
				\right)^{1 + \frac2{\alpha}}
			\\&
			+ \left( \log \frac{1}{t_0}\right)^{-\frac{\alpha}{2(2 + \alpha)}\frac{\theta}{2}} t_0^\frac{\alpha}{2} [\vv^m f]_{C^\alpha_v([t_0/4,t_0])}.
	\end{split}
\ee
In order to remove the last term above, it suffices to apply \cite[Lemma~B.2]{HST2019rough} (which is analogous to \Cref{l.weight_interpolation} but stated for standard H\"older spaces) to obtain
\be
	\left( \log \frac{1}{t_0}\right)^{-\frac{\alpha}{2(2 + \alpha)}\frac{\theta}{2}} t_0^\frac{\alpha}{2} [\vv^m f]_{C^\alpha_v([t_0/4,t_0])}
		\lesssim \left( \log \frac{1}{t_0}\right)^{-\frac{\alpha}{2(2 + \alpha)}\frac{\theta}{2}} t_0^\frac{\alpha}{2} \left(\| D^2_vf\|_{L^{\infty,m-2}([t_0/4,t_0])}^\frac{\alpha}{2}
			+ 1\right)
\ee
We remind the reader that $\|f\|_{L^{\infty,k}} \lesssim 1$.  This concludes the proof.
\end{proof}

\subsection{Proof of technical lemmas}\label{s.technical}

We begin by establishing the H\"older regularity of the transformed coefficients $\bar A$ and $\bar C$.  In order to make the notation more compact, we define, for any $z$,
\be\notag
	\tilde z
		:= z_0 \circ (Sz)_{r_0}
		= (r_0^2 t + t_0,
			r_0^3 S x + x_0 + r_0^2 t v_0,
			r_0 S v + v_0).
\ee
As $r_0$ and $z_0$ remain fixed in the following proof, there is no risk of confusion.

\begin{proof}[Proof of \Cref{l.w01161}]
We note that the proofs for $\bar C$ and $\bar A$ are essentially the same.  Hence, we show only the proofs of $(x,v)$-regularity of $\bar C$ and omit the proof of $x$-regularity of $\bar A$.  We include the proof of $v$-regularity of $\bar A$ in order to show how to extract $v$-regularity of $\bar A$ without using the $v$-regularity of $f$.

The proof of $x$-regularity is essentially the same as in \cite[Lemma~2.7]{HST2019rough}.  However, since our statement is a bit different (here there is no loss of regularity and we have less strict requirements on $m$), we provide the proof for completeness.

	We begin by fixing any $z, z' \in Q_1$ with $t = t'$ and $v = v'$.  Then
	\begin{equation}\notag
	\begin{split}
	|\bar C(z)-\bar C(z')|
	&\lesssim	
	\int |w|^{\gamma}|f(\tilde t, \tilde x, \tilde v-w)-f(\tilde t, \tilde x', \tilde v-w)|\,dw
	\\&\lesssim
		\int |w|^{\gamma} |r_0^3(Sx - Sx')|^{\alpha/3}
		\langle \tilde v - w\rangle^{-m} \|\vv^m f\|_{C^{\alpha/3}_x([t_0/2,t_0]\times \R^6)}\,dw.
	\end{split}
	\ee
Recalling the definitions of $r_0$ and $S$ in~\eqref{e.w05201}-\eqref{e.w05201} and that $|v| \leq 1$, we notice that
\be\label{e.c51409}
	\langle \tilde v - w\rangle^{-m}
		= \langle r_0 S v + v_0 - w\rangle^{-m}
		\lesssim \langle v_0 - w\rangle^{-m}.
\ee
Additionally,
\be\notag
	|\tilde x - \tilde x'|^{\alpha/3}
		= |r_0^3 (Sx - Sx')|^{\alpha/3}
		\lesssim t_0^{\alpha/2} |x-x'|^{\alpha/3}.
\ee
Hence,
\be\notag
	\begin{split}
	|\bar C(z) - \bar C(z')|
		&\lesssim  t_0^{\alpha/2} |x-x'|^{\alpha/3}
		\int |w|^{\gamma} 
		\langle \tilde v - w\rangle^{-m} \|\vv^m f\|_{C^{\alpha/3}_x([t_0/2,t_0]\times \R^6)}\,dw
		\\&\lesssim
		t_0^{\alpha/2} |x-x'|^{\alpha/3}\|\vv^m f\|_{C^{\alpha/3}_x([t_0/2,t_0]\times \R^6)} \langle \tilde v \rangle^{\gamma}
		\\&\lesssim
		t_0^{\alpha/2} |x-x'|^{\alpha/3}\|\vv^m f\|_{C^{\alpha/3}_x([t_0/2,t_0]\times \R^6)} \vvo^{\gamma}.
	\end{split}
\ee
In the last line, we used that $\langle \tilde v\rangle \approx \vvo$.  This concludes the proof of $x$-regularity for $\bar C$.  The proof of $v$-regularity is similar.

Now we establish the $v$-regularity of $\bar A$.  Let $z,z'\in Q_1$ with $t=t'$ and $x=x'$.  Changing variables, we have
\begin{equation}\notag
	\begin{split}
	\bar C(z)-\bar C(z')
		= c_\gamma \int (|w|^\gamma - |w+\tilde v' - \tilde v|^\gamma) f(\tilde t, \tilde x, \tilde v - w) \, dw.	
	\end{split}
\end{equation}
Let $R = 2|\tilde v - \tilde v'|$ and decompose the integral into two parts:
\be\notag
	|\bar C(z) - \bar C(z')|
		\lesssim \left(\int_{B_R} + \int_{B_R^c}\right)||w|^\gamma - |w+\tilde v' - \tilde v|^\gamma| f(\tilde t, \tilde x, \tilde v - w) \, dw
		= I_1 + I_2.
\ee

For $I_1$, notice that (recall the definitions of $r_0$ and $S$ in~\eqref{e.w05201}-\eqref{e.w05202} and that $|v|, |v'| \leq 1$)
\be\label{e.c51406}
	\langle \tilde v - w\rangle^{-m}
		= \langle r_0 S v + v_0 - w\rangle^{-m}
		\lesssim \langle v_0 - w\rangle^{-m}.
\ee
On the domain of $I_1$, clearly $\langle v_0 - w\rangle \approx \vvo$.  Hence,
\be\label{e.c51407}
	\begin{split}
	I_1
		&\lesssim \int_{B_R} (|w|^\gamma+|w + \tilde v' - \tilde v|^\gamma) \langle \tilde v - w\rangle^{-k} \|f\|_{L^{\infty,k}} \, dw
		\lesssim \vvo^{-k} \|f\|_{L^{\infty,k}} R^{3+\gamma}
		\\&\lesssim \vvo^{-k} \|f\|_{L^{\infty,k}} |\tilde v- \tilde v'|^{3+\gamma}.
	\end{split}
\ee

We now consider the final integral $I_2$.  Using again~\eqref{e.c51406}, we find
\be\label{e.c51408}
	\begin{split}
		I_2
			&\lesssim \int_{B_R^c}
				|w|^{\gamma - \alpha}
					|\tilde v - \tilde v'|^\alpha
					\langle v_0 - w\rangle^{-k}
					\|f\|_{L^{\infty,k}} \, dw
			\\&\leq
				|\tilde v - \tilde v'|^\alpha \|f\|_{L^{\infty,k}}
			 \int
				|w|^{\gamma - \alpha}
					\langle v_0 - w\rangle^{-k}
				\, dw
			\lesssim |\tilde v - \tilde v'|^\alpha \|f\|_{L^{\infty,k}}
				\vvo^{\gamma - \alpha}.
	\end{split}
\ee
In the third inequality we used that $\gamma - \alpha > -3$ so that the integral is finite.

Combining~\eqref{e.c51407},~\eqref{e.c51408}, and recalling the definitions of $r_0$ and $S$ (see~\eqref{e.w05201}-\eqref{e.w05202}), we arrive at
\be\notag
	|\bar C(z) - \bar C(z')|
		\lesssim t_0^{\alpha/2} \vvo^{\gamma - \alpha} \|f\|_{L^{\infty,k}} |v-v'|^{\alpha},
\ee
which concludes the proof.
\end{proof}

We next prove the $\log$-H\"older interpolation lemma. 
\begin{proof}[Proof of \Cref{l.w01081}]
	We begin by obtaining a bound on $\|D u\|_{L^\infty(Q_r)}$.  Let $v_0 \in Q_r$ be a point such that
	\be\label{e.c51401}
		\|D u\|_{L^\infty(Q_r)}
			\leq 2 |D u(v_0)|.
	\ee
	We claim that there is $\bar v$ so that
	\be\label{e.c51402}
		v_0 + \eps\bar v \in Q_r,
		\quad
		|\bar v| = 1,
		\quad\text{ and }\quad
		|\bar v \cdot Du(v_0)| \gtrsim |Du(v_0)|.
	\ee
	This is a basic (though somewhat complicated) plane geometry exercise that we postpone to the end of the proof.

	A Taylor expansion at $0$ yields, for some $\theta \in [0,1]$,
	\be\notag
		u(v_0 + \eps \bar v) - u(v_0)
			= \eps \bar v \cdot D u(v_0)
				+ \frac{\eps^2}{2} \bar v \cdot D^2 u(v_0 + \theta \eps\bar v) \bar v.
	\ee
	Rearranging this, recalling~\eqref{e.c51401} and~\eqref{e.c51402}, and dividing by $\eps$, we arrive at
	\be\label{e.c51403}
		\begin{split}
		\|D u\|_{L^\infty(Q_r)}
			&\leq 2 |D u(v_0)|
			\lesssim \frac{|u(v_0 + \eps \bar v) - u(v_0)|}{\eps} 
				+ \eps |D^2 u(v_0+\theta \eps \bar v)|
			\\&
			\leq \frac{\log(1/\eps)^{-\theta}}{\eps}
				[u]_{\log(\frac{1}{C})^{-\theta}(Q_r)} + \eps \|D^2 u\|_{L^\infty(Q_r)}.
		\end{split}
	\ee
	
	With~\eqref{e.c51403} in hand, we now use interpolation to obtain a bound on $D^2_v u$.  Indeed, using standard interpolation estimates (see, e.g., \cite[Proposition 2.10]{imbert2018schauder}), we have
	\be\notag
		\|D^2u\|_{L^\infty(Q_r)}
			\lesssim \left(\frac{\eps}{\delta}\right)^\alpha [D^2 u]_{C^\alpha(Q_r)}
				+ \frac{\delta}{\eps} \|D u\|_{L^\infty(Q_r)},
	\ee
	where $\delta>0$ is a parameter to be chosen.  Combining this with~\eqref{e.c51403}, we find
	\be\notag
		\begin{split}
			\|D^2u\|_{L^\infty(Q_r)}
				\lesssim \left(\frac{\eps}{\delta}\right)^\alpha [D^2 u]_{C^\alpha(Q_r)}
				+ \frac{\delta}{\eps} \left( \frac{\log(1/\eps)^{-\theta}}{\eps}
				[u]_{\log(\frac{1}{C})^{-\theta}(Q_r)} + \eps \|D^2 u\|_{L^\infty (Q_r)}\right).
		\end{split}
	\ee
	After choosing $\delta$ sufficiently small, depending only on the implied constant, we may absorb the $\|D^2 u\|_{L^\infty}$ term from the right hand side into the left hand side.  This yields
	\be\notag
		\begin{split}
			\|D^2u\|_{L^\infty(Q_r)}
				\lesssim \eps^\alpha [D^2 u]_{C^\alpha(Q_r)}
				+ \frac{\log(1/\eps)^{-\theta}}{\eps^2}
				[u]_{\log(\frac{1}{C})^{-\theta}(Q_r)},
		\end{split}
	\ee	
	which concludes the proof up to establishing~\eqref{e.c51402}.

	We now prove~\eqref{e.c51402}. At the expense of a multiplicative constant, we may assume that $\eps < r/10$.  
	Without loss of generality, we may assume that
	\be\label{e.c51414}
		\frac{Du(v_0)}{|Du(v_0)|} \cdot v_0 \leq 0.
	\ee
	Were this not the case, we work with $- Du(v_0)/|Du(v_0)|$ instead.  Then, we let
	\be\notag
		\bar v =  \frac{1}{10} \frac{Du(v_0)}{|Du(v_0)|}
				- \mu v_0,
	\ee
	where $\mu$ is chosen so that $|\bar v| = 1$.  Clearly, due to~\eqref{e.c51414},
	\be\label{e.c51415}
		|v_0| \mu \in [9/10,1].
	\ee
	Notice that
	\be\notag
		\bar v \cdot D u(v_0)
			= \frac{1}{10}|Du(v_0)| - \mu v_0 \cdot \frac{Du(v_0)}{|Du(v_0)|}
			\geq \frac{1}{10}|Du(v_0)|,
	\ee
	where the second inequality holds due to~\eqref{e.c51414}.  Next, using~\eqref{e.c51414} again as well as the fact that $\eps < r/10$,
	\be\notag
		|v_0 + \eps \bar v|
			= \Big|(1 - \eps \mu)v_0 + \frac{\eps}{10} \frac{Du(v_0)}{|Du(v_0)|}\Big|
			\leq |1 - \eps \mu| |v_0|
				+ \frac{\eps}{10}.
	\ee
	Consider the case when $\eps \mu \geq 1$, then, using~\eqref{e.c51415}
	\be\notag
		|v_0 + \eps \bar v|
			\leq \eps \mu |v_0|
				+ \frac{\eps}{10}
			\leq \frac{11 \eps}{10}
			< r.
	\ee
	which implies that $v_0 + \eps \bar v \in Q_r$.
	
	Next consider the case when $\eps \mu < 1$.  Then
	\be\notag
		|v_0 + \eps \bar v|
			\leq |v_0| - \frac{9 \eps}{10}
				+ \frac{\eps}{10}
			< |v_0|
			< r,
	\ee
	which again implies that $v_0 + \eps \bar v \in Q_r$.  Thus, we have established~\eqref{e.c51402}, which concludes the proof.
\end{proof}

We now prove the final technical lemma, \Cref{l.logH_scaling}, which involves the time scaling of the $\log$-H\"older norm of $f_{z_0}$, defined in~\eqref{e.f_z_0}.
\begin{proof}[Proof of \Cref{l.logH_scaling}]
Fix any $z\neq \tilde z \in Q_1$, with $t = \tilde t$, and notice that
\be\notag
	(r_0^2 t + t_0, r_0^3 S x + x_0, r_0 S v + v_0),
		(r_0^2 \tilde t + t_0, r_0^3 S \tilde x + x_0, r_0 S \tilde v + v_0)
		\in Q_{t_0/2}(z_0).
\ee
Hence,
\be\notag
	|f_{z_0}(z) - f_{z_0}(\tilde z)|
		\leq (r_0^\alpha |Sx - S\tilde x |^{\alpha/3} + \log(1/|r_0 (Sv-S\tilde v)|)^{-\theta})
			[f]_{\logalpha(Q_{t_0/2}(z_0))}.
\ee
From the definition of $S$, it is clear that
\be\notag
	r_0^3|S(x-\tilde x)|
		\lesssim t_0^{3/2}  |x-\tilde x|
	\quad\text{ and }\quad
	r_0|S(v-\tilde v)|
		\leq \sqrt{t_0} |v-\tilde v|.
\ee
Hence,
\be\label{e.c51405}
	\frac{|f_{z_0}(z) - f_{z_0}(\tilde z)|}{[f]_{\logalpha(Q_{t_0/2}(z_0))}}
		\lesssim  t_0^{\alpha/2} |x - \tilde x|^{\alpha/3}
			+ \left(\log\frac{1}{\sqrt{t_0}} + \log\frac{1}{|v-\tilde v|}\right)^{-\theta}.
\ee
Young's inequality yields
\be\notag
	\left(\log\frac{1}{\sqrt{t_0}} + \log\frac{1}{|v-\tilde v|}\right)^{-\theta}
		\lesssim
			\left(\log\frac{1}{\sqrt{t_0}}\right)^{-\theta/2}\left(\log\frac{1}{|v-\tilde v|}\right)^{-\theta/2}
\ee
and, it is straightforward to see that
\be\notag
	t_0^{\alpha/2}
		\lesssim \left(\log\frac{1}{t_0}\right)^{-\theta/2}.
\ee
Returning to~\eqref{e.c51405}, we find
\be\notag
	\frac{|f_{z_0}(z) - f_{z_0}(\tilde z)|}{[f]_{\logalpha(Q_{t_0/2}(z_0))}}
		\lesssim \left(\log\frac{1}{t_0}\right)^{-\theta/2}
			\left( |x-\tilde x|^{\alpha/3} + \left(\log\frac{1}{|v-\tilde v|}\right)^{-\theta/2}\right),
\ee
which concludes the proof.
\end{proof}

\begin{appendix}

\section{Computation of the fundamental solution~\eqref{e.gamma_a}.}\label{appendix}
In this section, we establish the form of the fundamental solution $\Gamma_{\bar a}$ for the $(x,v)$-homogeneous  kinetic Fokker-Planck equation; that is, we prove \Cref{p.gamma_a}.
\begin{proof}[Proof of \Cref{p.gamma_a}]
	We first notice that it is enough to find $\Gamma_{\bar a}$ such that the solution to the initial value problem
	\be\label{e.c050101}
		(\partial_t + v\cdot\nabla_x) f
			= \tr(\bar a(t) D_v^2 f),
	\ee
	with suitably decaying initial data at $t=\tilde t$ is given by
	\be\label{e.c050102}
		f(t,x,v)
			= \int_{\R^d}\int_{\R^d} \Gamma_{\bar a}(t,x - \tilde x - (t-\tilde t) \tilde v, v - \tilde v;\tilde t) f(\tilde t,\tilde x, \tilde v)\, d\tilde x d \tilde v.
	\ee
	Indeed, it is simply an application of Duhamel's principle to go from~\eqref{e.c050102} to~\eqref{e.gamma_a}.  As $\tilde t$ plays essentially no role in the computations below, we simply set $\tilde t = 0$ and drop the ``$;0$'' notation.

	Next, we notice that~\eqref{e.c050102} is equivalent to
	\be\label{e.c050201}
		\hat f(t,\xi,\omega)
			= (2\pi)^d \hat f(0,\xi, \omega + \xi t) \hat \Gamma_{\bar a}(t,\xi,\omega).
	\ee
	Indeed, 
	taking the Fourier transform of~\eqref{e.c050102} and computing, we find
		\begin{equation}\label{e.w05021}
		\begin{split}
		&\hat f(t,\xi, \omega)
		=
		\frac{1}{(2\pi)^d}
		\int_{\R^d}\int_{\R^d}
		\left(
		\int_{\R^d}\int_{\R^d} \Gamma_{\bar a}(t,x - \tilde x - t \tilde v, v - \tilde v) f(0,\tilde x, \tilde v)\, d\tilde x d \tilde v
		\right)\exp{\{-ix\cdot \xi -iv\cdot \omega\} }\,dxdv
		\\&=
		\frac{1}{(2\pi)^d}
		\int_{\R^d}\int_{\R^d}
		\int_{\R^d}\int_{\R^d} \Gamma_{\bar a}(t,x - \tilde x - t \tilde v, v - \tilde v) f(0,\tilde x, \tilde v)
		\exp{\{-ix\cdot \xi -iv\cdot \omega\} }\, d\tilde x d \tilde vdxdv.
		\end{split}
		\end{equation}
		As shifts in `physical space' correspond to multiplication in `Fourier space', we have
		\be\notag
			\hat \Gamma_{\bar a}(t,\xi,\omega)
				= \frac{e^{-i (\tilde x + t\tilde v)\cdot \xi - i\tilde v \cdot \omega}}{(2\pi)^d}\int_{\R^d}\int_{\R^d}
					\Gamma_{\bar a}(t,x-\tilde x - t\tilde v, v - \tilde v) \exp\{- ix \cdot \xi - i v \cdot \omega\} \, dx dv.
		\ee
		Thus,
		\be\notag
			\begin{split}
			\hat f(t,\xi,\omega)
				&= \hat \Gamma_{\bar a}(t,\xi,\omega) \int_{\R^d} \int_{\R^d} f(0,\tilde x, \tilde v) \exp\{- i(\tilde x+ t\tilde v)\cdot \xi - i \tilde v\cdot \omega\} \, d\tilde x d\tilde v
				\\&
				= \hat \Gamma_{\bar a}(t,\xi,\omega) \int_{\R^d} \int_{\R^d} f(0,\tilde x, \tilde v) \exp\{- i\tilde x\cdot \xi - i \tilde v\cdot (\omega+\xi t)\} \, d\tilde x d\tilde v
				\\&
				= (2\pi)^d \hat\Gamma_{\bar a}(t,\xi,\omega) \hat f(0, \xi, \omega + \xi t).
			\end{split}
		\ee

	We now find $\Gamma_{\bar a}$ through the identity~\eqref{e.c050201}.  The first step is to take the Fourier transform of~\eqref{e.c050101} in $x$ and $v$ to obtain:
	\begin{equation}\notag
	\begin{split}
	\partial_{t} \hat f-\xi\cdot \nabla_{\omega}\hat f
	=
	-\bar a(t)|\omega|^{2}\hat f
	,
	\end{split}
	\end{equation}
	Next, letting $\hat F(t,\xi,\omega)=\hat f(t,\xi, \omega-\xi t)$,  we have
	\begin{equation}\notag
	\begin{split}
	\partial_{t} \hat F
	=
	-(\omega-\xi t)^{T} \bar a(t)(\omega-\xi t)\hat F.
	\end{split}
	\end{equation}
	Integrating this in time, we find 
	\begin{equation}\notag
	\begin{split}
	\hat F(t,\xi, \omega)
	&=
	\exp \left\{-\int_{0}^{t}(\omega-\xi s)\cdot \bar a(s)(\omega-\xi s)\,ds\right\}
	\hat F(0,\xi, \omega)
	\\&=
	\exp\left\{-\int_{0}^{t}(\omega-\xi s)\cdot \bar a(s)(\omega-\xi s)\,ds\right\}
	\hat f(0,\xi, \omega).
	\end{split}
	\end{equation}
	Therefore,
	\begin{equation}\label{e.c42501}
	\begin{split}
	\hat f(t,\xi, \omega)
	=
	\hat f(0,\xi, \omega+\xi t)
	\exp\Big\{-\int_{0}^{t}(\omega-\xi (s-t))\cdot \bar a(s)(\omega-\xi (s-t))\,ds\Big\}.
	\end{split}
	\end{equation}
	It follows from~\eqref{e.c050201} that
	\be\label{e.c050202}
		\hat \Gamma_{\bar a}(t,\xi,\omega)
			= \frac{1}{(2\pi)^d} \exp\Big\{ - \int_0^t (\omega - \xi(s-t))\cdot \bar a(s) (\omega - \xi(s-t)) \, ds \Big\}.
	\ee
	
	The remainder of the proof is in computing the inverse Fourier transform of~\eqref{e.c050202}.  We begin by computing that:
	\begin{equation}\label{e.c050203}
	\begin{split}
	\Gamma_{\bar a} (t,x, v)
	&=
	\frac{1}{(2\pi)^{2d}}
	\int \int
	e^{-\int_{0}^{t}(\omega-\xi (s-t))\cdot \bar a(\omega-\xi (s-t))\,ds+i x\cdot \xi+ i v\cdot \omega}
	\,d\omega d\xi
	\\&=
	\frac{1}{(2\pi)^{2d}}
	\int e^{-\xi\cdot(A_{2}-2tA_{1}+t^{2}A)\xi+  i x\cdot \xi} \left(\int
	e^{-\omega\cdot A\omega +(2A_1 \xi  -2 tA_0\xi -i v)\omega}
	\,d\omega\right)d\xi
	\\&=
	\frac{1}{(2\pi)^{2d}}
	\int e^{-\xi\cdot N_2(t)\xi+  i x\cdot \xi} \bar \Gamma_{\bar a}\,d\xi
	,
	\end{split}
	\end{equation}
	where we have introduced the notation
	\be\notag
		\begin{split}
		&N_1(t,\xi,v) = 2 A_1(t) \xi - 2 t A_0(t) \xi - i v,
			\qquad
		N_2(t) = A_2(t) - 2t A_1(t) + t^2 A_0(t),
		\\&\text{and }\quad
		\bar \Gamma_{\bar a}(t,\xi,v)
			= \int e^{-\omega \cdot A_0(t) \omega + N_1(t) \cdot \omega} \, d\omega.
		\end{split}
	\ee
%
%
%
	We simplify $\bar \Gamma_{\bar a}$ by completing the square:
	\begin{equation}\notag
	\begin{split}
	\bar \Gamma_{\bar a} (t,\xi, v)
	&=
	\int
		e^{
			-(\omega-\frac{1}{2}A_0^{-1}N_1)\cdot A_0(\omega-\frac{1}{2}A_0^{-1}N_1)
			+\frac{1}{4}N_1 \cdot A_0^{-1}N_1}
		\,d\omega
	\\&=
	e^{\frac{1}{4}N_1 \cdot A_0^{-1}N_1}
	\int
	e^{-(\omega^{}-\frac{1}{2}A_0^{-1}N_1) \cdot A_0(\omega^{}-\frac{1}{2}A_0^{-1}N_1)}
	\,d\omega
	=
	e^{\frac{1}{4}N_1\cdot A_0^{-1}N_1}\frac{\pi^{d/2}}{\sqrt{\det{A_0}}}.
	\end{split}
	\end{equation}
	Plugging this into~\eqref{e.c050203} and then completing the square for the $\xi$-integral, we find
	\begin{equation}\notag
	\begin{split}
	\Gamma_{\bar a} (t,x, v)
	&= 
	\frac{1}{2^{2d}\pi^\frac{3d}{2} \sqrt{\det{A_0}}}
	\int e^{ - \xi \cdot N_2 \xi +  i x\cdot \xi+\frac{1}{4}N_1\cdot A_0^{-1}N_1} 
	\,d\xi
	\\&=	\frac{1}{2^{2d}\pi^\frac{3d}{2} \sqrt{\det{A_0}}}
	e^{-\frac{v\cdot A_0^{-1}v}{4} - \frac{1}{4} q \cdot P^{-1} q}
	\int
	e^{-(\xi^{}-\frac{i}{2}P^{-1}q)\cdot P(\xi^{}-\frac{i}{2}P^{-1}q)}
	\,d\xi,
	\end{split}
\end{equation}
	where (recall $M$ from~\eqref{e.A_i})	
	\begin{equation}\notag
	\begin{split}
	&P
	=
	N_2-(tA_0-A_1)A^{-1}(tA_0-A_1)
	=
	A_{2}-A_{1}A_0^{-1}A_1
	\\&\quad\text{and }\quad
	q
	=
	x - vt + A_1 A_0^{-1} v
	=
	x-Mv.
	\end{split}
	\end{equation}
	Computing the the integral and simplifying, we find
	\be\notag
		\Gamma_{\bar a}(t,x,v)
			= \frac{1}{2^{2d} \pi^d \sqrt{\det(A_0) \det(P)}}
			e^{-\frac{v\cdot A_0^{-1}v}{4} - \frac{1}{4} q \cdot P^{-1} q}.
	\ee
	This concludes the proof.
\end{proof}

\section{Interpolation of weights between $L^{\infty,k}$ and $\logalpha$}

\begin{Lemma}\label{l.weight_interpolation}
	Fix any $\alpha, \mu \in (0,1)$ and any $\theta, k > 0$.  Suppose that
	\be\notag
		\varphi \in L^{\infty,m}(\R^3) \cap \logvalpha(\R^3).
	\ee
	Then $\vv^{(1-\mu)k} \varphi \in  \log\left(1/C_v\right)^{-\theta\mu}$ and
	\be\notag
		[\vv^{(1-\mu) k} \varphi]_{ \log\left(1/C_v\right)^{-\theta\mu}}
			\lesssim \|\varphi\|_{L^{\infty,k}}^{1-\mu}
				[\varphi]_{\logvalpha}^{\mu}
				+ \|\varphi\|_{L^{\infty, ((1-\mu)k -1)_+}}.
	\ee
\end{Lemma}
	\begin{proof}
		First, for $(t,v)\neq(t,v') \in \R_+ \times \R^3$ with $|v-v'|<1/2$, we let 
		\begin{equation}\notag
		R=\vv^{-k}\|\phi\|_{L^{\infty, k}}[\varphi]^{-1}_{\logvalpha}.
		\end{equation}
		Then, we obtain
		\begin{equation}\notag
		\begin{split}
		|\vv^{(1-\mu)k}&\varphi(t,v) -\vvp^{(1-\mu)k}  \varphi(t,v')|\\
		&\lesssim
		\vv^{(1-\mu)k}|\varphi(t,v) - \varphi(t,v')|
		+
		|\varphi(t,v')|
		|\vv^{(1-\mu)k} -\vvp^{(1-\mu)k}  |
		\\&
		\lesssim
		\vv^{(1-\mu)k}|\varphi(t,v) - \varphi(t,v')|
		+
		|\varphi(t,v')|
		\vv^{((1-\mu)k-1)_+}|v-v'|
		\\&\lesssim
		\vv^{(1-\mu)k}|\varphi(t,v) - \varphi(t,v')|
		+
		\|\varphi\|_{L^{\infty,((1-\mu)k-1)_+}}
		|v-v' |
		.
		\end{split}
		\end{equation}
		Notice that
		\be\notag
			\frac{
		|v-v' |}{\log(1/|v- v'|)^{-\mu\theta}}
				\lesssim 1.
		\ee
		Hence, we need only bound
		\be\notag
			H := \frac{\vv^{(1-\mu)k}|\varphi(t,v) - \varphi(t,v')|}{\log(1/|v- v'|)^{-\mu\theta}}.
		\ee
		If $\log(1/|v- v'|)^{-\theta} \geq R$, we have
		\begin{equation}\notag
		\begin{split}
		&H
		\leq
		2\vv^{-\mu k}\frac{\|\varphi\|_{L^{\infty,k}}}{R^\mu}
		=2\|\varphi\|_{L^{\infty,k}}^{1-\mu}[\varphi]^{\mu}_{\logvalpha},
		\end{split}
		\end{equation}
		which is the desired inequality.
		
		On the other hand, if $\log(1/|v- v'|)^{-\theta} \leq R$, we see
		\begin{equation}\notag
		\begin{split}
		H
		&=
		\vv^{(1-\mu)k}\frac{| \varphi(t,v) -  \varphi(t,v')|}{\log(1/|v- v'|)^{-\theta}}
		(\log(1/|v- v'|)^{-\theta})^{1-\mu}
		\leq \vv^{(1-\mu)k}[\varphi]^{}_{\logvalpha}
		R^{1-\mu}
		\\&\lesssim
		[\varphi]^{}_{\logvalpha}
		\|\phi\|^{1-\mu}_{L^{\infty, k}}[\varphi]^{\mu-1}_{\logvalpha}
		=
		[\varphi]^{\mu}_{\logvalpha}
			\|\phi\|^{1-\mu}_{L^{\infty, k}},
		\end{split}
		\end{equation}
		which is, again, the desired inequality.  This concludes the proof.
	\end{proof}

\end{appendix}

\section*{Acknowledgments}
The authors would like to sincerely thank Marco Bramanti and Andrea Pascucci for helpful discussions regarding their preprints~\cite{BiagiBramanti,LPP}.  CH was partially supported by NSF grants DMS-2003110 and DMS-2204615. WW was partially supported by an AMS-Simons travel grant.  CH thanks Otis Chodosh for first bringing the references~\cite{Brandt,Knerr} to his attention.

\bibliographystyle{abbrv}
\bibliography{schauder_landau_HW2}

\end{document}